\documentclass[10pt]{amsart}
\usepackage{amsthm,amsmath,amsxtra,amscd,amssymb,xypic,cleveref,color,mathtools,cleveref}
\usepackage[all]{xy}

\newcommand{\bfe}{{\bf e}}
\newcommand{\bfn}{{\bf n}}
\newcommand{\bfN}{{\bf N}}

\def\bm#1{\text{\boldmath$#1$}}

\newcommand{\de}{\delta}
\newcommand{\ep}{\varepsilon}

\newcommand\nl{\mathrm{null}}
\newcommand\dec{\mathrm{dec}}

\newcommand\tc[2]{\theta\chars{#1}{#2}}
\newcommand\tnull{{\theta_{\operatorname{null}}}}
\newcommand\chars[2]{\left[\begin{smallmatrix}#1\\ #2\end{smallmatrix}\right]}

\newcommand{\CC}{{\mathbb{C}}}

\newcommand{\HH}{{\mathbb{H}}}

\newcommand{\EE}{{\mathbb{E}}}
\newcommand{\PP}{{\mathbb{P}}}
\newcommand{\QQ}{{\mathbb{Q}}}
\newcommand{\RR}{{\mathbb{R}}}
\newcommand{\ZZ}{{\mathbb{Z}}}

\newcommand{\NN}{{\mathbb{N}}}

\newcommand{\calH}{{\mathcal H}}

\newcommand{\calI}{{\mathcal I}}
\newcommand{\calA}{{\mathcal A}}

\newcommand{\calL}{{\mathcal L}}
\newcommand{\calM}{{\mathcal M}}
\newcommand{\calJ}{{\mathcal J}}

\newcommand{\calX}{{\mathcal X}}

\newcommand{\calD}{{\mathcal D}}

\newcommand{\calS}{{\mathcal S}}

\newcommand{\fR}{{\mathfrak R}}

\newcommand{\hI}{{\widehat I}}

\newcommand{\bfp}{{\mathbf n'}}
\newcommand{\bfP}{{\mathbf N'}}

\newcommand{\DIFF}{\mathfrak{D}}
\newcommand\tP{\widetilde{P}}

\newcommand{\op}{\operatorname}

\newcommand{\Sat}[1][g]{{\calA_{#1}^*}}
\newcommand{\Perf}[1][g]{{\overline{\calA}_{#1}}}
\newcommand{\Part}[1][g]{{\calA'_{#1}}}

\newcommand{\Hom}{\op{Hom}}
\newcommand{\Sp}{\op{Sp}}

\newcommand{\GL}{\op{GL}}
\newcommand{\SL}{\op{SL}}

\newcommand{\Sym}{\op{Sym}}

\newcommand{\Pic}{\op{Pic}}

\newcommand{\even}{\op{even}}

\newcommand{\ord}{\op{ord}}

\newcommand{\Hess}{\op{Hess}}
\newcommand\sgn{\mathrm{sgn}}

\newcommand\Eff{{\operatorname{Eff}}}
\newcommand\Amp{{\operatorname{Amp}}}
\newcommand\Mov{{\operatorname{Mov}}}

\newcommand\cond{$\star$}

\newcommand{\pa}{\partial}

\theoremstyle{plain}
\newtheorem{thm}{Theorem}[section]
\newtheorem{lm}[thm]{Lemma}
\newtheorem{prop}[thm]{Proposition}
\newtheorem{cor}[thm]{Corollary}

\newtheorem{maintheorem}{Theorem}
\newenvironment{mainthm}[1]
{\renewcommand{\themaintheorem}{#1}
\begin{maintheorem}}
{\end{maintheorem}}

\newtheorem{maincorollary}{Corollary}
\newenvironment{maincor}[1]
{\renewcommand{\themaincorollary}{#1}
\begin{maincorollary}}
{\end{maincorollary}}

\theoremstyle{definition}
\newtheorem{df}[thm]{Definition}
\newtheorem{rem}[thm]{Remark}
\newtheorem{ntn}[thm]{Notation}
\newtheorem{exa}[thm]{Example}

\title[ Differentiating  modular forms and the moving slope of $\calA_g$]{Differentiating Siegel modular forms \\ and the moving slope of $\calA_g$}

\author[S. Grushevsky]{Samuel Grushevsky}
\address{Department of Mathematics and Simons Center for Geometry and Physics, Stony Brook University, Stony Brook, NY 11794-3651}
\email{sam@math.stonybrook.edu}

\author[T. Ibukiyama] {Tomoyoshi Ibukiyama}
\address{Department of Mathematics, Graduate School of Mathematics, Osaka University,
Machikaneyama 1-1, Toyonaka, Osaka, 560-0043 Japan}
\email{ibukiyam@math.sci.osaka-u.ac.jp}

\author[G. Mondello] {Gabriele Mondello}
\address{Universit\`a ``La Sapienza'', Dipartimento di Matematica Guido Castelnuovo, Piazzale A. Moro 2, I-00185, Roma,   Italy}
\email{mondello@mat.uniroma1.it}

\author[R. Salvati Manni]{Riccardo Salvati Manni}
\address{Universit\`a ``La Sapienza'', Dipartimento di Matematica Guido Castelnuovo, Piazzale A. Moro 2, I-00185, Roma,   Italy}
\email{salvati@mat.uniroma1.it}

\thanks{Research of the first author is supported in part by NSF grants DMS-18-02116 and DMS-21-01631. The work of the second author is supported by JSPS Kakenhi Grant Number JP19K03424. The third and fourth  authors were partially supported by PRIN 2017 grant ``Moduli spaces and Lie theory''.  The third author was also partially supported by INdAM GNSAGA research group.}
\date{\today}
\begin{document}

\begin{abstract}
We study the cone of moving divisors on the moduli space $\calA_g$ of principally polarized abelian varieties. Partly motivated by the generalized Rankin-Cohen bracket, we construct a non-linear holomorphic differential operator that sends Siegel modular forms to Siegel modular forms, and we apply it to produce new modular forms. Our construction recovers the known divisors of minimal moving slope on $\calA_g$ for $g\leq 4$, and gives an explicit upper bound for the moving slope of $\calA_5$ and a conjectural  upper bound for the moving slope of $\calA_6$.
\end{abstract}

\maketitle

\section{Introduction}

\subsection{Moduli of ppav and compactifications}
Denote $\calA_g$ the moduli space of complex principally polarized abelian varieties (ppav), which is the quotient of its (orbifold) universal cover, the Siegel upper half-space $\HH_g$, by the action of the symplectic group $\Sp(2g,\ZZ)$. Let $\Sat$ denote the Satake-Baily-Borel compactification, and recall that the Picard group $\Pic_\QQ(\Sat)=\QQ \lambda$ is one-dimensional, generated by the class~$\lambda$ of the line bundle $\calL\rightarrow\Sat$ of Siegel modular forms of weight one, which is ample on~$\Sat$.

Let $\Part$ be Mumford's partial compactification of $\calA_g$, so that
$\pa\Part=\calX_{g-1}/\pm 1$, where
$\pi:\calX_{g-1}\to\calA_{g-1}$ denotes the universal family of ppav of dimension~$g-1$.

All toroidal compactifications of $\calA_g$ contain $\Part$. The boundary of the perfect cone toroidal compactification $\Perf$ (we use this notation as no other toroidal compactification will appear) is an irreducible Cartier divisor~$D$, and $\pa\Part$ is dense within $D$. The compactification $\Perf$ is $\QQ$-factorial, with $\Pic_\QQ\Perf=\QQ\lambda\oplus \QQ\delta$, where
$\delta$ denotes the class of~$D$. The Picard group $\Pic_\QQ\Part$ is generated by the restrictions of the classes $\lambda$ and $\delta$ from $\Perf$ to $\Part$. 
Philosophically, in what follows, the definition of the slope of divisors takes place on $\Part$, though to formally make sense of it we work on $\Part$ (and refer to \cite[Appendix]{gsmPrym} for a discussion of why this notion is the same for any other toroidal compactification).

\subsection{The ample and effective slopes}
Given a divisor $E$ on $\Perf$ such that its class in the Picard group is $[E]=a\lambda-b\delta$, its {\em slope} is defined to be $s(E)\coloneqq a/b$. The slope of a cone of divisors on $\Perf$ is defined as the infimum of the slopes of divisors contained in the cone.
Shepherd-Barron \cite{sb} proved that the {\em ample slope} of $\calA_g$, that is the slope of the cone of ample divisors
is equal to $12$, namely
$$
 s_\Amp(\Perf)\coloneqq \inf\left\{s(E)\colon E\in \Amp(\Perf)\right\}=12\,.
$$

The {\em effective slope}, that is the slope of the cone of effective divisors
$$
 s_\Eff(\Perf)\coloneqq \inf\left\{s(E)\colon E\in \Eff(\Perf)\right\}\, ,
$$
has attracted a lot of attention, in particular because
\[
s(K_{\Perf})=s\left((g+1)\lambda-\delta\right)=g+1\,,
\]
so that the inequality $s_\Eff(\Perf)<g+1$ would imply that $\calA_g$ is of general type. Freitag \cite{frei} used the theta-null divisor $\tnull$, of slope $s(\tnull)=8+2^{3-g}$, to show that $\calA_g$ is of general type for $g\ge 8$, Mumford \cite{mum} used the Andreotti-Mayer divisor $N_0$, of slope $s(N_0)=6+\tfrac{12}{g+1}$, to show that $\calA_g$ is of general type for $g\ge 7$, while recently the fourth author with collaborators \cite{mss} showed that $s_\Eff(\Perf[6])\leq 7$, which implies that the Kodaira dimension of $\calA_6$ is non-negative.

It is known that $\calA_g$ is unirational for $g\leq 5$ (see \cite{MoMu}, \cite{C}, \cite{D},  \cite{V} for the harder cases of $g=4,5$). In fact $s_\Eff(\Perf)$ is known explicitly for $g\leq 5$: the computation of $s_\Eff(\Perf[5])$ is one of the main results of \cite{fgsmv}, and the lower genera cases are reviewed below.

On the other hand, the slope $s_\Eff(\Perf)$ is not known for any $g\ge 6$. While the techniques of Tai \cite{tai} show that $s_\Eff(\Perf)=O(1/g)$ for $g\to \infty$ (as explained in~\cite{grag}), not a single explicit effective divisor~$E$ on $\calA_g$, for any~$g$, with $s(E)\leq 6$ is known.

The analogous notion of effective slope for the moduli space of curves $\calM_g$ has been
investigated in many papers, in particular for its similar link with the Kodaira dimension of $\calM_g$, starting with \cite{harrismumford}
\cite{harris}, \cite{eisenbudharris}, and with continuing recent progress such as \cite{fjp}.

\subsection{The moving slope}\label{sec:intro-moving}
Recall that an effective divisor $E$ is called {\em moving} if $h^0(E)>1$ and if moreover the base locus of its linear system $|E|$ has codimension at least two.
The {\em moving slope} is the slope of the cone $\Mov$ of moving divisors
$$
 s_\Mov(\Perf)\coloneqq \inf\{s(E)\colon E\in \Mov(\Perf)\}\,.
$$
Since the moving cone is contained in the effective cone, we have $s_\Eff(\calA_g)\leq s_\Mov(\calA_g)$. We first observe that if the effective slope is in fact an infimum but not a minimum, then $s_\Eff(\calA_g)=s_\Mov(\calA_g)$ since there is an infinite sequence of effective divisors of strictly decreasing slopes converging to this infimum (see \Cref{lm:properties}(iii) for a precise statement and proof).
Thus investigating the moving slope is only of interest if there exists an effective divisor $E\subset\calA_g$ of slope $s(E)=s_\Eff(\calA_g)$.

\smallskip

While the moving slope of $\Perf$ is less well-studied than the effective slope, it is also important in attempting to determine the structure of the ring of Siegel modular forms, and in attempting to run the log-MMP for $\calA_g$ and determine its interesting birational models:
in fact, the pull-back of an ample
divisor on a normal projective variety $X$ via a non-constant rational map
$f:\Perf\dashrightarrow X$ is a moving divisor, as remarked in \cite[Section 1.2]{cfm}.

The moving slope of $\Perf$ is known for $g\leq 4$, as we will review below, and Tai's results also imply that $s_\Mov(\Perf)=O(1/g)$ as $g\to\infty$. While the original published version of the paper \cite{fgsmv_published} claimed an upper bound for $s_\Mov(\Perf[5])$, there was a numerical error, and the corrected (arXiv) version \cite{fgsmv} does not allow to deduce any statement on $s_\Mov(\Perf[5])$.
For $g=6$ the knowledge of the moving slope of $\Perf[6]$ would help
determining the Kodaira dimension of $\calA_6$, if it turns out that $s_\Eff(\Perf[6])=7=s(K_{\Perf[6]}$).
As in the case $g=5$, though, the moving slope of $\Perf$ remains unknown at present for every $g\geq 6$.

\subsection{Context}
Our paper revolves around the problem of constructing, from a given modular form, or from given modular forms, new modular forms of controlled slope. In particular, given a modular form of minimal slope, such procedure can provide other interesting modular forms of low slope: for example, for $2\leq g\leq 4$ it does provide a modular form of minimal moving slope (\Cref{maincor:mov}). Our construction(s) will consist
in applying certain holomorphic differential operators to Siegel modular forms, so as to yield Siegel modular forms again (\Cref{thm:main}).

For motivation, recall the definition of two such well-known operators for $g=1$.
The first one is the {\em Serre derivative} (credited by Serre \cite[Theorem 4]{serre:modulaire}
to Ramanujan \cite{ramanujan}): it sends Siegel modular forms
of weight $a$ to modular forms of weight $a+2$, and is defined as
$\calS_a(F):=\frac{d F}{d \tau}-\frac{\pi ia}{6}\,E_2\cdot F$, where $E_2$ is the Eisenstein series
of weight $2$ (see also \cite[Section 5]{zagier} and \cite[Lemma 3]{sd}).
The second one is the second {\em Rankin-Cohen bracket} (see \cite{rankin} and \cite{cohen}), which sends a modular form of weight $a$
to a modular form of weight $2a+4$, and is defined as $[F,F]_{2,a}:=aF\frac{d^2 F}{d\tau^2}- (a + 1)\left(\frac{d F}{d \tau}\right)^2$.
Note that $\calS_a$ is a $1$-homogeneous (i.e.~multiplying $F$ by a constant $\lambda$ multiplies $\calS_a(F)$ by $\lambda^1$)
differential operator in $\tau$  with non-constant
coefficients, while $[\cdot,\cdot]_{2,a}$ is $2$-homogeneous, of pure order $2$
(meaning that all summands involve
the derivative $\frac{d}{d\tau}$ twice),
with constant coefficients.
There are also $2n$-th Rankin-Cohen brackets $[\cdot,\cdot]_{2n,a}$, which
are $2n$-homogeneous, of pure order $2n$, with constant coefficients,
and send modular forms of weight $a$ to modular forms of weight $2a+4n$.

The holomorphic differential operators that we will produce
for $g\geq 2$ are, on one hand, analogous to $\calS_a$,
as they will be $g$-homogeneous, of order $g$; on the other hand,
they share some similarities with the even Rankin-Cohen brackets,
as they will be pure of order~$g$ (meaning that each summand
involves exactly $g$ partial derivatives),
with constant coefficients.

\subsection{Main results}

In order to formulate our main result, given a holomorphic function $F:\HH_g\to\CC$, we assemble the coefficients of its differential $dF$
into the matrix
\[
\pa  F\coloneqq\left(\begin{array}{cccc}
 \frac{\pa  F}{\pa  \tau_{11}} & \frac{\pa  F}{2\pa  \tau_{12}}&\dots& \frac{\pa  F}{2\pa  \tau_{1g}} \\
\vdots & \vdots&\ddots&\vdots \\
  \frac{\pa  F}{2\pa  \tau_{g1} }& \frac{\pa  F}{2\pa  \tau_{g2}}&\dots& \frac{\pa  F}{\pa  \tau_{gg}}
 \end{array}\right)\,,
\]
and we consider the holomorphic function $\det(\pa  F):\HH_g\to\CC$.

Suppose now that $F$ is a modular form of weight $a$, with vanishing order $b$ along the boundary $\pa\Perf$ (this will be defined formally in the next section). The determinant $\det(\pa  F)$ is in general not a modular form, but its restriction to the zero locus $\{F=0\}$ behaves as
a modular form of weight $ga+2$ (a more intrinsic approach to $\det(\pa  F)$ will be given in \Cref{rem:intrinsic}). Our main result is the following construction.

\begin{mainthm}{A}\label{thm:main}
For every $g\geq 2$ and every integer $a\geq \frac{g}{2}$ there exists a
differential operator $\DIFF_{g,a}$
acting on the space of genus $g$ Siegel modular forms of weight $a$ that satisfies the following properties:
\begin{itemize}
\item[(i)]
if $F$ is a genus $g$ Siegel modular form of weight $a$ and vanishing order $b$ along the boundary, then $\DIFF_{g,a}(F)$ is a Siegel modular form of weight $ga+2$
and of vanishing order $\beta\ge gb$ along the boundary;
\item[(ii)]
the restriction of $\DIFF_{g,a}(F)$ to the zero locus of $F$ is equal to the restriction of $\det(\pa  F)$.
\end{itemize}
\end{mainthm}

\begin{rem}
In \Cref{thm:main} it is possible to deal with Siegel modular forms $F$ with character with respect to $\Sp(2g,\ZZ)$, which occur only for $g=2$ only. Since $\DIFF_{2,a}$ is quadratic, $\DIFF_{2,a}(F)$ will then still be a modular form (with trivial character).
\end{rem}

What we will actually prove is a more precise version of this statement.
In \Cref{thm:precise} we construct
for every $g\geq 2$ and $a\geq \frac{g}{2}$
a holomorphic differential operator $\calD_{Q_{g,a}}$ in the $\tau_{ij}$
with constant coefficients
and we define $\DIFF_{g,a}(F):=\calD_{Q_{g,a}}(F)/g!$ for every
Siegel modular forms $F$ of genus $g$ and weight $a$.
Thus $\DIFF_{g,a}(F)$ is always polynomial in~$F$ and its partial derivatives,
though its coefficients depend on the weight $a$.
Though the operator $\DIFF_{g,a}$ need not be unique,
properties (i-ii) in \Cref{thm:main}
force the Siegel modular form
$\DIFF_{g,a}(F)$ to be unique up to adding modular forms divisible by~$F$ (which would thus vanish on the zero locus of~$F$).
The construction is explicit, and in \Cref{sec:explicit} we will give the formulas for $\DIFF_{2,a}$ and $\DIFF_{3,a}$ explicitly.

\smallskip
A priori $\DIFF_{g,a}(F)$ could have a common factor with $F$, or could even be identically zero. In order to prevent such behavior,
we will apply \Cref{thm:main} only to modular forms $F$ that satisfy what will be our {\em main condition:}
\begin{equation}\tag{\cond}\label{cond}
\begin{array}{c}
 \hfill \text{\it $\det(\pa  F)$ does not vanish identically }\hfill\\
 \hfill\text{\it on any irreducible component of $\{F=0\}$.}\hfill
 \end{array}
\end{equation}

Our main application is an immediate consequence of \Cref{thm:main}.
\begin{maincor}{B}\label{cor:main}
Suppose that the effective slope $s_\Eff(\Perf)=a/b$ is realized by
a modular form $F$ of weight $a\geq\frac{g}{2}$
that satisfies Condition \eqref{cond}. Then
\begin{equation}\label{eq:ineqEffMov}
s_\Mov(\Perf)\leq s(\DIFF_{g,a}(F))\leq s_\Eff(\Perf)+\tfrac{2}{bg}\,.
\end{equation}
\end{maincor}

We first note that if the zero locus of $F$ in \Cref{cor:main} is not irreducible,
then every its irreducible component must have the same slope. As a consequence, $s_\Mov(\Perf)=s_\Eff(\Perf)$ (as will be proven carefully in \Cref{lm:properties}(i)), and so the statement becomes trivial. Thus we can assume that the zero locus of $F$ is irreducible.

We stress that the inequality \eqref{eq:ineqEffMov} for the moving slope depends on the actual class $[F]$, not just on the slope~$s(F)$. Moreover, Condition \eqref{cond} forces $F$ to be square-free.  For every $g\leq 5$ it is known that a reduced effective divisor on $\Perf$ of minimal slope
exists and is unique. For $g\leq 4$ the machinery
of \Cref{cor:main} produces an (already known) divisor
that realizes the moving slope.

\begin{maincor}{C}\label{maincor:mov}
For $2\leq g\leq 4$ the modular form $F$ of minimal slope on $\Perf$
satisfies Condition \eqref{cond} and has weight $a\geq \frac{g}{2}$.
Moreover, $\DIFF_{g,a}(F)$ realizes the moving slope of $\calA_g$.
\end{maincor}

For $g=5$, in \cite{fgsmv} it was proven that the Andreotti-Mayer divisor $N'_0$ (whose definition will be recalled in \Cref{sec:N0}) is the
unique effective divisor of minimal slope on $\Perf[5]$. Since we will show in \Cref{prop:cond} that $N'_0$ satisfies Condition \eqref{cond},
as a consequence of \Cref{cor:main} we obtain
\begin{maincor}{D}\label{cor:g=5}
The moving slope of $\calA_5$ is
bounded above by $s_\Mov(\Perf[5])\leq \tfrac{271}{35}$,
and the slope $271/35$ is achieved by a moving effective divisor.
\end{maincor}

In the following table, we collect what is thus known about the effective and moving slopes of
$\Perf$:
\begin{center}
\begin{tabular}{c|c|c}
 & $s_{\Eff}(\Perf)$ & $s_{\Mov}(\Perf)$ \\
\hline
$g=1$ & $12$  & \\
$g=2$ & $10$ & $12$ \\
$g=3$ & $9$ & $28/3=9.333\dots$ \\
$g=4$ & $8$  & $17/2=8.500\dots$ \\
$g=5$ & $54/7=7.714\dots$ & $\leq\,271/35=7.742\dots$ \\
$g=6$ & $[\frac{53}{10},\ 7]$ & (?) $\leq\,43/6=7.166\dots$\\
$g\gg 1$ & O(1/g) & O(1/g)
\end{tabular}
\end{center}
where the upper bound $s_\Eff(\Perf[6])\leq 7$ is provided by the Siegel modular form $\theta_{L,h,2}$ of class $14\lambda-2\delta$ constructed in \cite{mss}. The question mark in the above table marks a conjectural upper bound $s_\Mov(\Perf[6])\leq 43/6$, which
is a consequence of the following.

\begin{maincor}{E}\label{cor:genus6}
The form $\theta_{L,h,2}$ on $\calA_6$ is prime, i.e. not a product of non-constant Siegel modular forms.
Moreover, if $\theta_{L,h,2}$ satisfies Condition \eqref{cond},
then
\[
s_\Mov(\Perf[6])\leq \frac{43}{6}.
\]
\end{maincor}

The Torelli map $\tau_g:\calM_g\to\calA_g$ sending a curve to its Jacobian is an injection of coarse moduli spaces, but for $g\ge 3$ is 2-to-1 as a map of stacks. We denote by $\calJ_g$ the closure of $\tau_g(\calM_g)$ inside $\calA_g$, which is called the locus of Jacobians. For $g\leq 3$ $\calJ_g=\calA_g$, while $\calJ_4\subset\calA_4$ is the zero locus of the Schottky modular form $S_4$, which has weight $8$.
Since (even) theta constants always vanish on curves with even multiplicity, this implies that $\tnull\cap\calJ_g=2\Theta_{\nl}$ for $g\ge 3$, where $\Theta_{\nl}\subset\calJ_g$ is an integral divisor. As a byproduct of our analysis, we also obtain the following result on Jacobians:

\begin{maincor}{F}\label{cor:M4}
The form $\DIFF_{4,8}(S_4)$ restricts on $\calJ_4$ to $\Theta_{\nl}$.
\end{maincor}

Beyond these results, we investigate the applications of both Rankin-Cohen brackets and of differential operators acting on Siegel modular forms to constructing new effective divisors.
Our results above go essentially one step in this direction, by applying the differentiation technique to a modular form of lowest slope. This construction can be iterated or varied to apply it to a tuple of different modular forms: it would be interesting to investigate the collection of modular forms thus produced, and to see in particular if this sheds any further light on the structure of the ring of Siegel modular forms in any genus $g\ge 4$, where it is not fully known.

\subsection{Structure of the paper}
The paper is organized as follows. In \Cref{sec:modforms} we set the notation and review the relation between effective divisors on~$\calA_g$ and Siegel modular forms. In \Cref{sec:relevant} we recall the construction and the slopes of the theta-null divisor $\tnull$ and of the Andreotti-Mayer divisor $N'_0$, and we show that both satisfy Condition \eqref{cond}. In \Cref{sec:RC} we define the Rankin-Cohen bracket and prove a weaker version of \Cref{cor:main}. In \Cref{sec:lowgenus} we review the computation of the effective and moving slopes for $g\leq 4$, derive
Corollaries \ref{maincor:mov}-\ref{cor:g=5}-\ref{cor:genus6} from \Cref{thm:main}, and prove \Cref{cor:M4}. Finally, in \Cref{sec:diffop} we introduce
a remarkable class of differential operators acting on Siegel modular forms, we
define $\calD_{Q_{g,a}}$ and we prove \Cref{thm:main}.

\section*{Acknowledgements}
We are grateful to Gavril Farkas and Alessandro Verra for discussions related to \cite{fgsmv_published}, leading to the correction \cite{fgsmv} and to a better understanding of the intricacies of the geometry involved. We are indebted to Claudio Procesi for the proof of \Cref{procesi}. S.~Grushevsky thanks the Weizmann Institute of Science for hospitality in spring 2022, when this work was done and the paper was written.


\section{Siegel modular forms and compactifications of $\calA_g$}\label{sec:modforms}
We briefly recall the standard notions on Siegel modular forms, referring to \cite{frei} for a more detailed introduction. Unless specified otherwise, we assume $g\geq 2$.

\subsection{The Siegel space and the moduli space of ppav}
The Siegel upper half-space $\HH_g$ is the space
of complex symmetric $g\times g$ matrices $\tau$
with positive definite imaginary part.

An element $\gamma$ of the symplectic group $\Sp(2g,\ZZ)$,
written as $\gamma=\left(\begin{smallmatrix}A & B \\ C & D \end{smallmatrix}\right)$ in $g\times g$ block form, acts on $\HH_g$ via
\[
\gamma\cdot \tau\coloneqq (A\tau+B)(C\tau+D)^{-1}.
\]
The action of $\Sp(2g,\ZZ)$ on $\HH_g$ is properly discontinuous, with finite
stabilizers. The quotient $\calA_g=\HH_g/\Sp(2g,\ZZ)$ is the moduli space of ppav ---
it is a quasi-projective variety that can be given the structure of an orbifold
(or a Deligne-Mumford stack).
We denote by $\pi:\calX_g\to\calA_g$ the universal family of ppav.

\subsection{Divisors and Siegel modular forms}
A holomorphic   function $F:\HH_g\to\CC$ is called a holomorphic   {\em Siegel modular form of weight $k$} with respect to $\Sp(2g,\ZZ)$ if
\[
F(\gamma\cdot\tau)=\det(C\tau+D)^kF(\tau)
\]
for all $\tau\in\HH_g$ and for all $\gamma\in\Sp(2g,\ZZ)$
(for $g=1$ there is an additional regularity condition).

This automorphy property with respect to $\Sp(2g,\ZZ)$ defines the line bundle
\[
\calL^{\otimes k}\longrightarrow\calA_g
\]
of Siegel modular forms of weight $k$  on $\calA_g$.

\begin{rem}
While in our paper we focus on Siegel modular forms for $\Sp(2g,\ZZ)$, the holomorphic differential operator that we consider is defined for any holomorphic functions on~$\HH_g$, and will preserve suitable automorphy properties. It can thus also be applied to Siegel modular forms with  multiplier systems for subgroups of $\Sp(2g,\ZZ)$.  In  particular we will apply it to  a Siegel modular forms with a character, namely the theta-null $T_2$ in genus two, discussed in \Cref{sec:theta}.
\end{rem}

\subsection{Satake compactification}
The Satake-Baily-Borel compactification $\Sat$ can be defined as
\[
\Sat\coloneqq \mathbf{Proj}\left(\oplus_{n\geq 0} H^0(\calA_g,\calL^{\otimes n})\right).
\]
What this means is that sections of a sufficiently high power of $\calL$ embed $\calA_g$ into a projective space, and $\Sat$ is the closure of the image of $\calA_g$ under such an embedding. Since $\Pic_\QQ(\Perf)=\QQ\lambda$, where $\lambda$ denotes the class of $\calL$,
this implies that any effective $\ZZ$-divisor on $\calA_g$ is the zero locus of a Siegel modular form.

\subsection{Partial and perfect cone toroidal compactifications}
Set-theoretically, $\Sat$ is the union of locally closed strata
\[
\Sat=\calA_g\sqcup \calA_{g-1}\sqcup\dots\sqcup \calA_0.
\]
The partial (aka Mumford, or rank one) toroidal compactification
\[
\Part\coloneqq\calA_g\sqcup\pa\Part
\]
is obtained by blowing up the partial Satake compactification $\calA_g\sqcup\calA_{g-1}$
along its boundary $\calA_{g-1}$, and the exceptional divisor $\pa\Part$ can then be identified with $\calX_{g-1}/\pm 1$.

Any toroidal compactification contains $\Part$ and admits a blowdown morphism to $\Sat$. The perfect cone toroidal compactification $\Perf$ has the property that the complement $\Perf\setminus\Part$ is of codimension $2$ inside $\Perf$. The boundary
\[
D\coloneqq\pa \Perf
\]
is an irreducible Cartier divisor, which is the closure of $\pa\Part$. We denote by $p:\Perf\to\Sat$ the blowdown map.

\subsection{Effective divisors on $\calA_g$}
The effective and moving slope is computed on effective divisors in $\calA_g$,
or, equivalently, on effective divisors in $\Perf$,
whose support does not contain~$D$. We will call such divisors {\em internal}.
For clarity and completeness, we explain how to associate an internal divisor to a Siegel modular form.

A Siegel modular form $F$ of weight $a$, thought of as a section of $\calL^{\otimes a}$ on $\Sat$, can be pulled back to a section of $p^*\calL^{\otimes a}$ on~$\Perf$.
If the vanishing order $\ord_D(p^*F)$ of $p^*F$ along the divisor~$D$ is $b$,
this means that the zero locus of $p^*F$ on $\Perf$
is the union of an effective divisor not containing~$D$ in its support,
which we will denote by $(F)$ and call the zero divisor of the modular form, and of the divisor~$D$ with multiplicity $b$.
Since by definition the zero locus $\{ F=0\}\subset\Sat$ has class $a\lambda$, its preimage in $\Perf$ has class $ap^*\lambda$ (or $a\lambda$ in our notation abuse), it follows that the class of the zero divisor of a modular form is
\[
[(F)]=a\lambda-b\delta\in \Pic_\QQ(\Perf)
\]
with $a>0$ and $b\geq 0$.

To summarize the above discussion, we see that internal effective divisors on $\Perf$ correspond bijectively to Siegel modular forms up to multiplication by a constant, and from now on we will talk about them interchangeably, additionally suppressing the adjective ``internal'' as we will never need to deal with effective divisors on $\Perf$ whose support contains~$D$.

We thus define the slope $s(F)$ of a modular form $F$ to be the slope of the corresponding (internal) effective divisor $(F)$. We will write $F$ for the modular form considered on $\Perf$, and stress  that the notation $[F]\coloneqq [(F)]$ for the class of the zero divisor of a Siegel modular form on $\Perf$ does {\em not} signify the class of the pullback $p^*F$, which would be simply equal to $a\lambda$.

Every effective divisor $E\subset\Perf$ can be uniquely written as $E=\sum c_i E_i$
for suitable $c_i>0$ and pairwise distinct,
irreducible, reduced divisors $E_i$. We say that two divisors $E=\sum c_iE_i$ and $E'=\sum d_j E'_j$ have
distinct supports if $E_i\ne E'_j$ for all $i,j$.

Similarly, a Siegel modular form $F$
can be uniquely written as a product $F=\prod F_i^{c_i}$ for suitable
$c_j>0$ and pairwise distinct, {\em prime} Siegel modular forms $F_i$
(i.e.~forms that cannot be factored as products of non-constant modular
forms). Two modular forms $F=\prod F_i^{c_i}$
and $F'=\prod (F'_j)^{d_j}$ are said to {\em not have a common factor}
if $F_i\neq F'_j$ for all $i,j$.

\subsection{Fourier-Jacobi expansion}
The vanishing order of a Siegel modular form~$F$ at~$D$ can be computed using the Fourier-Jacobi expansion, which we  briefly recall for further use.
Writing an element $\tau\in\HH_g$ as
  $$\tau = \left( \begin{matrix} \tau' & z \\ z^t & w \end{matrix} \right) \in \HH_g$$
   with $ \tau' \in\HH_{g-1}, \,\,z\in\CC^{g-1}, \,\, w\in \CC^*$,
 and setting  $q\coloneqq \exp(2\pi i w)$, we expand~$F$ in power series in $q$:
\begin{equation}\label{F-J}
 F(\tau) = \sum_{r\geq 0} f_r(\tau', z) q^r.
 \end{equation}
Then the  vanishing order $\ord_DF$ (which we will often denote~$b$) of $F$ along~$D$ is detected by the Fourier-Jacobi expansion as
\begin{equation}
 \ord_D F=\min\{ r\geq 0\: f_r(\tau',z) \not\equiv 0\}\, .
\end{equation}
The form $F$ is called a {\em cusp form} if it vanishes identically on~$D$; equivalently, if $f_0(\tau',0)=0$, that is if $\ord_DF>0$.

\subsection{First properties of the moving slope}
Here we record some properties of the moving slope, showing that one should only focus on the case when there exists an effective divisor of minimal slope, and furthermore that one should only focus on irreducible effective divisors. These are general properties that we state for $\Perf$, but hold on any projective variety.

\begin{lm}\label{lm:properties}
The moving slope satisfies the following properties:
\begin{itemize}
\item[(i)]
if $E\neq E'$ are irreducible reduced effective divisors, then $$s_\Mov(\Perf)\leq \max\{s(E),s(E')\}\,;$$
\item[(ii)]
if $s_\Mov(\Perf)=s(E)$ for {\em some} moving divisor $E$, then
there exists an {\em irreducible} moving divisor $E'$
such that $s_\Mov(\Perf)=s(E')$;
\item[(iii)]
if there does {\em not} exist an effective divisor $E$ such that $s(E)=s_\Eff(\Perf)$, then $s_\Eff(\Perf)=s_\Mov(\Perf)$.
\end{itemize}
\end{lm}
\begin{proof}
(i) Let $[E]=a\lambda-b\delta$ and $[E']=a'\lambda-b'\delta$, and suppose that
$s(E)\leq s(E')$.
Then the linear system $|aE'|$ contains $a'E$, and its base locus
is contained inside $E\cap E'$, which has codimension at least two.
It follows that $aE'$ is a moving divisor.
Since $[aE']=a(a'\lambda-b'\delta)$, we obtain
$s_\Mov(\Perf)\leq s(aE')=s(E')=a'/b'$.

(ii) If the general element of the linear system $|E|$ is irreducible,
then we can choose $E'$ to be any such element.
Otherwise a general element $E_t\in |E|$ can be written as a sum
$E_t=E^1_t+\dots+E^m_t$ of $m$ distinct effective divisors. Moreover each
$E^i_t$ belongs to a linear system whose base locus is contained in the base locus of $E_t$. Hence each $E^i_t$ is moving.
Since $s_\Mov(\Perf)\leq \min_i s(E^i_t)\leq s(E_t)=s_\Mov(\Perf)$,
we conclude that $s(E^i_t)=s_\Mov(\Perf)$ for all $i$.
Hence, it is enough to take $E'=E^i_t$ for any $i$.

(iii)
Consider a sequence $(E_n)$ of effective divisors on $\calA_g$ whose slopes are strictly decreasing and converging to $s_\Eff(\Perf)$.
Up to replacing $E_n$ by the irreducible component of $E_n$ with smallest slope,
and up to passing to a subsequence, we can assume that all $E_n$ are irreducible.
Since the slopes are strictly decreasing, the $E_n$ are all distinct.
Applying (i) to the pair $E_{n-1},E_n$, we have $s_\Mov(\Perf)\leq s(E_n)$.
The conclusion follows,
since $s_\Eff(\Perf)\leq s_\Mov(\Perf)\leq s(E_n)\rightarrow s_\Eff(\Perf)$.
\end{proof}


\section{Some relevant modular forms}\label{sec:relevant}
In this section we briefly recall the definitions and the main properties of theta constants,
of the Schottky form, and of Andreotti-Mayer divisors.

\subsection{Theta functions and theta constants}\label{sec:theta}
For $\ep,\de\in \{0,1\}^g$ the {\em theta function
with characteristic $\chars\ep\de$} is the function
$\tc\ep\de:\HH_g\times\CC^g\rightarrow\CC$
defined as
$$
\tc\ep\de(\tau,z)\coloneqq\sum\limits_{n\in\ZZ^g} \exp \pi i \left[\left(
n+\tfrac{\ep}{2}\right)^t\tau \left(n+\tfrac{\ep}{2}\right)+2\left(n+\tfrac{\ep}{2}\right)^t\left(z+
\tfrac{\de}{2}\right)\right].
$$

Characteristics $\chars\ep\de$ are called {\em even} or {\em odd} depending on
the parity of the standard scalar product $\langle\ep , \de\rangle$. This is the same as the parity of $\theta$ as a function of $z\in\CC^g$ for fixed $\tau\in\HH_g$, and there are $ 2^{g-1}( 2^{g}+1)$ even characteristics and $2^{g-1}( 2^{g}-1)$ odd ones.  The {\em theta constant} is the evaluation of the theta function at $z=0$, which is thus a function $\tc\ep\de(\tau):\HH_g\to\CC$. By the above, all odd theta constants vanish identically, while an even theta constants are modular form of weight $1/2$ (meaning that a suitable square root of the automorphic factor $\det(C\tau+D)$ is taken)
with respect to a certain finite index subgroup of $\Sp(2g,\ZZ)$. The product of all even theta constants
$$
  T_g\coloneqq\prod_{\chars\ep\de\, \even}\theta\chars\ep\de
$$
turns out to be a modular form for the full symplectic group, for $g\ge 3$, called the {\em theta-null} modular form, and its zero locus is called the theta-null divisor $\tnull$. It has class
\begin{equation}\label{eq:classTg}
  [T_g]=2^{g-2}(2^g+1)\lambda-2^{2g-5}\delta\,,\qquad \hbox{and so }\quad s(T_g)=s(\tnull)=8+2^{3-g}\,.
\end{equation}
The case $g=2$ is slightly different since $T_2$ has a character, meaning that it satisfies
\[
T_2(\gamma\cdot\tau)=\pm\det(C\tau+D)^5 T_2(\tau)
\]
for all $\gamma=\left(\begin{smallmatrix} A & B\\ C & D\end{smallmatrix}\right)\in
\Sp(4,\ZZ)$. Hence $T_2^2$ is a well-defined modular form
and we still have $[T_2]=5\lambda-\delta/2$ in $\Pic_\QQ(\Perf)$.

\subsection{The Schottky form}\label{sec:schottky}
The {\em Schottky form} is the weight $8$ modular form on $\calA_g$ given by the following degree 16 polynomial in theta constants:
\[
S_g\coloneqq  \frac{1}{2^g}\sum_{\ep,\de}\theta^{16}\chars\ep\de - \frac{1}{2^{2g}}\left(\sum_{\ep,\de}\theta^8\chars\ep\de\right)^2\,.
\]
The Schottky form is a modular form for $\Sp(2g,\ZZ)$, and is natural because it can alternatively be expressed as $S_g=\theta_{D_{16}^+}-\theta_{E8\oplus E8}$ as the difference of the lattice theta functions associated to the only two even, unimodular lattices in $\RR^{16}$ (see \cite{igusagen4} or \cite{igusachristoffel}). It is known that $S_g$ vanishes identically on $\calA_g$ if and only if $g\leq 3$, and moreover that the zero locus of $S_4$ is the locus of Jacobians $\calJ_4\subset\calA_4$. The form $S_4$ vanishes identically to first order along $D$, and thus
\begin{equation}
[S_4]=8\lambda-\delta\,,
\quad\text{and}\quad
s(S_4)=8\,,
\end{equation}
while for $g\ge 5$ the form $S_g$ is not a cusp form (and so it has infinite slope).

\subsection{Andreotti-Mayer divisor}\label{sec:N0}
The {\em Andreotti-Mayer divisor} \cite{anma} is defined to be the locus $N_0$ of ppav whose theta divisor is singular.




It is known that $N_0$ is a divisor that has for $g\ge 4$ precisely two irreducible components: $N_0=\tnull\cup N'_0$ (see \cite{mum},\cite{debarredecomposes}), while for $g=2,3$ the Andreotti-Mayer divisor is simply $N_0=\tnull$.

\begin{rem}\label{gen-sing}
For a generic point of $\tnull$, the unique singularity of the theta divisor of the corresponding ppav is the double point at the two-torsion point of the ppav corresponding to the characteristic of the vanishing theta constant. It is known that generically this singular point is an ordinary double point (i.e.~that the Hessian matrix, of the second derivatives of the theta function with respect to $z$ at this point is non-degenerate).
For a generic point of $N'_0$, the theta divisor of the corresponding ppav has precisely two opposite singular points, both of which are generically ordinary double points again, see \cite{gsmordertwo} for a detailed study.
\end{rem}

As we already know, $\tnull$ is the zero locus of the modular form $T_g$ that is the product of all even theta constants, and we know the class of the corresponding divisor by~\eqref{eq:classTg}. The modular form, which we denote $I_g$, defining the effective divisor $N_0'$ is not known explicitly for any $g\ge 5$ (see \cite{krsm}), while the Riemann theta singularity implies that in genus 4 we have $N'_0=\calJ_4$, and thus $I_4=S_4$. The class of the divisor $N'_0$ was computed by Mumford \cite{mum}:
%
%
%
\begin{equation}\label{eq:classAM}
[N'_0]=[I_g]=(g!(g+3)/4- 2^{g-3}(2^g +1)) \lambda - ((g+1)!/24   - 2^{2g-6}) \delta\,,
\end{equation}
\begin{equation*}
\text{and so}\quad s(I_g)=
6\cdot\frac{1+2/(g+1)- 2^{g-1}(2^g +1)/(g+1)!}{1   - 3\cdot 2^{2g-3}/(g+1)!}>6\, .
\end{equation*}

\smallskip
Even though an equation for $N'_0$ is not known, a precise description of its tangent space is provided by
the following Lemma, which is a special case of results proven in \cite{anma} (see also \cite{ACGH}).

\begin{lm}\label{AM}\label{lm:AM}
Let $Z$ be $\tnull$ or $N'_0$ and call $\widetilde{Z}$ its preimage in $\HH_g$.
For every general smooth point $\tau_0$ of $\widetilde{Z}$, and every
ordinary double point $z_0\in\CC^g$ of $\theta(\tau_0,\cdot)=0$,
the tangent space $T_{\tau_0}\widetilde{Z}$
has equation $d_\tau \theta(\tau_0,z_0)=0$
inside $T_{\tau_0}\HH_g$.
\end{lm}
Since theta functions satisfy the heat equation
\begin{equation}\label{eq:heat}
    d_\tau\theta=2\pi i\cdot\Hess_z\theta \, ,
\end{equation}
where $\Hess_z$ denotes the Hessian, that is the matrix of the second partial derivatives of the theta function with respect to $z_1,\dots,z_g$,
by \Cref{lm:AM}
the differentials $dT_g$ and $dI_g$ are related to the Hessian of the theta function in the $z$-variables.
We then have the following
\begin{prop}\label{prop:cond}
The form $T_g$  for $g\ge 2$  and the form $I_g$ for $g\ge 4$ satisfy Condition \eqref{cond}.
\end{prop}
The genus restrictions in this statement are simply to ensure that the forms are well-defined and not identically zero.
\begin{proof}
If $\tau_0$ is a smooth point of $\tnull$, then the theta divisor $\Theta_{\tau_0}\subset X_{\tau_0}=\CC^g/(\ZZ^g\oplus\tau_0\ZZ^g)$ is singular at a unique $2$-torsion point,
and such a singularity is ordinary if and only if $\det(dT_g)\neq 0$ at $\tau_0$
by \eqref{eq:heat} and \Cref{lm:AM}.

Similarly, if $\tau_0$ is a generic point of $N'_0$, then
the singular locus of $\Theta_{\tau_0}$ consists of
two opposite non-$2$-torsion singular points $\pm z_0$;
moreover, $\pm z_0$ are ordinary double points of $\Theta_{\tau_0}$ if and only if $\det(dI_g)\neq 0$ at $\tau_0$ by \eqref{eq:heat} and \Cref{lm:AM}.

The conclusion follows from \Cref{gen-sing}.
\end{proof}

\section{Rankin-Cohen bracket}\label{sec:RC}
Our method to bound the moving slope of $\calA_g$ from above is by constructing new Siegel modular forms starting from a given known modular form. For example, starting from the known Siegel modular form minimizing the slope of the effective cone, we will try to construct another Siegel modular form, with which it has no common factor, and which has a slightly higher slope. In this section we do this using the Rankin-Cohen bracket (of two different modular forms), which will allow us to prove the main application \Cref{cor:main}, but only under the assumption that the moving slope is achieved.

While our construction of the differential operators $\DIFF_{g,a}$ in \Cref{thm:main} yields a stronger result, we now give the details of the geometrically motivated construction using the Rankin-Cohen brackets. These were defined in \cite{rankin} and \cite{cohen} for $g=1$ (see also \cite{zagier}); a vector-valued version appears in \cite{satoh} and a scalar-valued version appears in \cite{yangyin}.

For further use, we define the symmetric $g\times g$ matrix-valued holomorphic differential operator acting on functions on~$\HH_g$
\begin{equation}\label{eq:padefined}
\pa_{\tau}\coloneqq \left(\frac{1+\delta_{ij}}{2}\frac{\pa}{\pa \tau_{ij}}\right)_{1\leq i,j\leq g}\,.
\end{equation}
When no confusion is possible, we will sometimes denote this differential operator simply by~$\pa$.

\subsection{Vector-valued bracket}
Let $F$ and $G$ be genus $g$ Siegel modular forms of weights $k$ and $h$ respectively.
\begin{df}[\cite{satoh}]
The {\em vector-valued Rankin-Cohen bracket} of $F$ and $G$ is
\[
\{F,G\}\coloneqq \frac{G^{k+1}}{F^{h-1}} \cdot d\left(\frac{F^h}{G^k}\right).
\]
where $d=d_\tau$ is the differential of a function of $\tau\in\HH_g$.
\end{df}

\begin{lm}\label{lemma:bracket}
The vector-valued bracket
\[
\{F,G\}=-\{G,F\}=hG\,dF-kF\,dG
\]
is a $\calL^{\otimes(h+k)}$-valued holomorphic $(1,0)$-form on $\calA_g$.
Moreover $\{F,G\}\not\equiv 0$ unless $F^h$ and $G^k$ are constant multiples of each other.
\end{lm}
\begin{proof}
Since $F^h/G^k$ is a meromorphic function on $\HH_g$, its differential is a meromorphic $(1,0)$-form. Moreover, $G^{k+1}/F^{h-1}$ is a meromorphic Siegel modular form of weight $h+k$ (i.e.~it is a meromorphic function on $\HH_g$ that satisfies the transformation property).
It is immediate to check that $\{F,G\}=hG\,dF-kF\,dG$,
which shows that $\{F,G\}$ is thus a holomorphic Siegel-modular-form-valued $(1,0)$ form. Since $F$ and $G$ are non-zero, the bracket vanishes identically if and only if $d(F^k/G^h)$ is identically zero, which is equivalent to this ratio being a constant.
\end{proof}

Another way to state \Cref{lemma:bracket} is that, writing $\{F,G\}$ as a
$g\times g$ matrix, this matrix satisfies the transformation law
\[
\{F,G\}( \gamma\cdot\tau)=  \det(C\tau+D)^{k+h} (C\tau+D)^t  \cdot\{F,G\}(\tau)\cdot (C\tau+D)
\]
for any $\gamma=\left(\begin{smallmatrix}A & B \\ C & D \end{smallmatrix}\right)$
in $\Sp(2g,\ZZ)$.

\subsection{Scalar-valued bracket}
Let $\EE\to\calA_g$ denote the holomorphic rank~$g$ Hodge bundle of $(1,0)$-holomorphic forms
on ppav, namely $\EE=\pi_*\Omega^{1,0}_\pi$ (where we recall that $\pi:\calX_g\to\calA_g$ denotes the universal family of ppav).
Recall that the cotangent space
$T^*\calA_g$ can be identified with $\Sym^2\EE\subset \Hom(\EE^\vee,\EE)$.
Since
$$
\det:\Hom(\EE^\vee,\EE)\to (\det\EE)^{\otimes 2}\subset
\Lambda^g(\Sym^2\EE)\cong\Omega^{g,0}\calA_g
$$
and $\det\EE\cong\calL$, it follows that $\det$ restricts
to a map $\det:T^*\calA_g\to \calL^{\otimes 2}$, which is homogeneous of degree $g$.
If $f$ is a meromorphic function defined on $\calA_g$,
then $\det(df)$ is a meromorphic section of $\calL^{\otimes 2}$.

\begin{df}
The {\em scalar Rankin-Cohen bracket} of  Siegel modular forms $F,G$ is defined as
\[
[F,G]\coloneqq\det\{F,G\}\, .
\]
\end{df}

The scalar Rankin-Cohen bracket
seems not to have been systematically studied in the literature.
Here we collect some of its basic properties.

\begin{lm}\label{lm:scalar-bracket}
Let $F,G$ be Siegel modular forms, of classes
$$
 [F]=k\lambda-x\delta;\qquad [G]=h\lambda-y\delta.
$$
Then $[F,G]$ is a Siegel modular form of class
$$\Big[[F,G]\Big]=g(k+h)\lambda-\beta\delta\,,$$
 where
\begin{itemize}
\item[(i)] $\beta> 0$ (i.e.~$[F,G]$ is a cusp form, even if $F$ and $G$ are not);
\item[(ii)] $\beta\ge g(x+y)$;
\item[(iii)] for any integer $n>0$, $[F,F^n]=0$;
\item[(iv)] if $H$ is another modular form, then $[H^2 F,G]$ and $[HF,HG]$ are divisible by $H^g$;
\item[(v)] if $F,G$ do not have any common factors, and $F$ satisfies Condition \eqref{cond}, then~$F$ and $[F,G]$ do not have any common factors.
\end{itemize}
\end{lm}
\begin{proof}
(i) Recall that $\{F,G\}=(G^{k+1}/F^{h-1})\cdot d\left(F^h/G^k\right)$
and so $\det\{F,G\}=G^{g(k+1)}/F^{g(h-1)}\det(d\left(F^h/G^k\right))$.
It follows that $\det\{F,G\}$ is a modular form of weight
$gh(k+1)-gk(h-1)+2=g(h+k)+2$ and, from the local expression of $\{F,G\}$,
it follows that $[F,G]$ is holomorphic.

Consider then the Fourier-Jacobi expansions
\[
F(\tau)=F_0(\tau',0)+\sum_{r>0}F_r(\tau',z)q^r,
\quad
G(\tau)=G_0(\tau',0)+\sum_{r>0}G_r(\tau',z)q^r,
\]
at $\tau=\left(\begin{smallmatrix}\tau' & z \\ z^t & w \end{smallmatrix}\right)$.
We have
\[
dF=\left(
\begin{array}{cc}
d_{\tau'}F & d_z F\\
(d_z F)^t & d_w F
\end{array}
\right),\quad
dG=\left(
\begin{array}{cc}
d_{\tau'}G & d_z G\\
(d_z G)^t & d_w G
\end{array}
\right).
\]
Recall that $q=\exp(2\pi iw)$, so that $\pa  (q^r)/\pa  w=2\pi r i q^r$.
It is immediate to check that the last columns of $dF$ and $dG$ are divisible by $q$.
It follows that $[F,G]$ is divisible by $q$, and so is a cusp form.

(ii) Writing $dF$ and $dG$ as above, it is immediate
that $\ord_D dF=\ord_D F$ and $\ord_D dG=\ord_D G$.
Hence $\ord_D \{F,G\}=\ord_D F+\ord_D G$, and the conclusion follows.

(iii) By direct computation $\{F,F^n\}=(nk) F^n dF-kF (n F^{n-1})dF=0$\,.

(iv) Let $\ell$ be the weight of $H$; we compute directly
$$
\begin{aligned}
\{H^2 F,G\}&=hG(H^2 \,dF+2HF\,dH)-(2\ell+k)H^2 F \,dG\\ &=
H(hHG\, dF+2hF\,dH-(2\ell+k)HF \,dG)
\end{aligned}
$$
and
$$
\begin{aligned}
\{HF,HG\}&=(\ell+h)HG(H\,dF+F\,dH)-(\ell+k)HF(H\,dG+G\,dH)\\ &=
H(H\{F,G\}+\ell H(G\,dF-F\,dG)+(h-k)FG\,dH)\,.
\end{aligned}
$$

(v)
Note first that Condition \eqref{cond} implies that $F$ is square-free.
Evaluating $[F,G]$ along the zero divisor of $F$, we obtain
\begin{equation}\label{eq:FGwhereF=0}
[F,G]|_{F=0}=h^gG^g\det(\pa  F)\,.
\end{equation}
Since $F$ and $G$ do not have common factors,
$[F,G]$ is identically zero
along a component of $\{F=0\}$ if and only if $\det(\pa  F)$ is.
\end{proof}
\begin{rem}
It is possible that the strict inequality $\beta>g(x+y)$ holds in (ii) above: for example,
(i) implies that $\beta\geq 1$ for $x=y=0$.
\end{rem}

\begin{rem}\label{rem:intrinsic}
Statement (v) above is one instance where we see the key importance of Condition ~\eqref{cond}, and of $\det(\pa F)$.
A more intrinsic description of the function $\det(\pa  F)$ is as follows.
If $F$ is a modular form of weight $k$,
its differential is not well-defined on $\calA_g$,
but the restriction of $dF$ to the zero divisor $E=\{F=0\}$ of $F$ is.
Thus $dF|_E$ is a section
of $\calL^{\otimes k}\otimes \Sym^2\EE|_E$, and
$\det(dF)$ is a section of $\calL^{\otimes (kg+2)}|_E$.
In other words, the restriction of $\det(\pa F)$ to
the zero locus of $F$ behaves as a modular form of weight $gk+2$, as mentioned in the introduction.
\end{rem}

\subsection{The bracket and the moving slope}\label{sec:bra-mov}
In this section we apply the scalar Rankin-Cohen bracket
to two modular forms of low slope in order to produce another modular form
of low slope. This will allow us to prove the following weaker version of \Cref{cor:main} --- it is
weaker only in that it assumes that the moving slope is achieved, i.e.~is a minimum rather than infimum.
\begin{prop}\label{prop11}
Assume that the effective slope $s_\Eff(\Perf)=a/b$ is realized by a Siegel modular form $F$
of class $a\lambda-b\delta$ that satisfies Condition \eqref{cond}.
Suppose moreover that the moving slope $s_\Mov(\Perf)=a'/b'$ is achieved by a Siegel modular form $G$
of class $a'\lambda-b'\delta$.
Then
\[
s_\Mov(\Perf)\leq s_\Eff(\Perf)+\frac{2}{bg}\, .
\]
\end{prop}
\begin{proof}
If $F$ is a product of at least two distinct prime factors, then
each of them realizes the effective slope.
Hence $s_\Mov(\Perf)=s_\Eff(\Perf)$ by \Cref{lm:properties}(i),
and so the conclusion trivially holds.
Hence we can assume that $F$ is a prime Siegel modular form.

Up to replacing $G$ by a general element in its linear system, we can assume that $F$ does not divide $G$.
By \Cref{lm:scalar-bracket}(v), the form $[F,G]$
is not divisible by $F$, and so in particular $[F,G]$ does not identically vanish.
It follows from \Cref{lm:scalar-bracket}(ii) that
\[
\frac{a'}{b'} = s_\Mov(\Perf)\leq s([F,G])\leq \frac{g(a+a')+2}{g(b+b')}\,,
\]
which can be rewritten as
\[
s_\Mov(\Perf)=\frac{a'}{b'}\leq \frac{a}{b}+\frac{2}{bg}=s_\Eff(\Perf)+\frac{2}{bg}\,.
\]
\end{proof}

Both the scalar Rankin-Cohen bracket and $\DIFF_{g,a}$ (which will be introduced in \Cref{sec:diffop})
are holomorphic differential operators of degree $g$,
but their relationship is not clear, and deserves a further investigation.


\section{Effective and moving slopes for small $g$}\label{sec:lowgenus}

In this section we recall what is known about the effective and moving
slopes of $\calA_g$ for $2\leq g\leq 5$.
In all these cases the effective slopes are achieved, and
we analyze what we obtain by applying
\Cref{thm:main} (whose proof is postponed till \Cref{sec:diffop})
to such effective divisors of minimal slope,
and we prove Corollaries \ref{maincor:mov}-\ref{cor:g=5}-\ref{cor:genus6}-\ref{cor:M4}.

\subsection{Case $g=2$}
In genus $2$ the unique effective divisor of minimal slope is the closure
of the locus $\calA_2^{\dec}$
of decomposable abelian varieties inside $\Perf[2]$. Set-theoretically, this locus is simply equal to the theta-null divisor $\tnull$.
We thus obtain
\[
s_\Eff(\Perf[2])=s(\calA_2^{\dec})= s(\tnull)=s(5\lambda-\delta/2)=10\,.
\]
\begin{rem}
Note that the class $[T_2]=\frac{1}{2}(10\lambda-\delta)$ in
$\Pic_\QQ(\Perf)$ is not integral, though its double is.
From the stacky point of view, this is a manifestation of the fact that
$\calA_2^{\dec}\cong (\calA_1\times\calA_1)/S_2$
and so the general element of $\calA_2^{\dec}$
has an automorphism group $\{\pm 1\}\times\{\pm 1\}$, of order $4$, whereas
the general genus $2$ ppav has automorphism group $\{\pm 1\}$, of order $2$.
\end{rem}

As mentioned in the introduction,
\Cref{thm:main} can be applied to $T_2$,
even though $T_2$ is a modular form with character.
Since $T_2$ satisfies Condition \eqref{cond} by \Cref{prop:cond},
we obtain a cusp form $\DIFF_{2,5}(T_2)$
of weight $12$ that is not identically zero on $\tnull$.
As in \Cref{cor:main}, it follows that
$$
 s_\Mov(\Perf[2])\leq s\left(\DIFF_{2,5}(T_2)\right)=12\,.
$$

\begin{proof}[Proof of \Cref{maincor:mov} for $g=2$]
It is known \cite{frei} that the ideal of cusp forms inside the ring of genus $2$
Siegel modular forms is generated by two modular forms
$\chi_{10}\coloneqq T_2^2$ and $\chi_{12}$, which
has class $[\chi_{12}]=12\lambda-\delta$.
It then follows that $\DIFF_{2,5}(T_2)$ and $\chi_{12}$ are proportional,
and so $\DIFF_{2,5}(T_2)$ realizes the moving slope.
\end{proof}

Since $T_2$ satisfies Condition \eqref{cond} by \Cref{prop:cond},
and since $T_2$ and $\chi_{12}$ are square-free and without common factors,
\Cref{lm:scalar-bracket}(v) ensures that
the cusp form
$[T_2,\chi_{12}]$ does not vanish identically along $\tnull$.
By \Cref{lm:scalar-bracket} it follows that
\[
[[T_2,\chi_{12}]]=36\lambda-3\delta\,,
\]
and so $[T_2,\chi_{12}]$ is another Siegel modular form that achieves the moving slope.\

\subsection{Jacobian and hyperelliptic loci}\label{sec:jac}

As mentioned in the introduction, it is possible to define
a slope for effective divisors in the moduli space
$\calM_g$ of projective curves of genus $g$. We denote
\[
\tau_g : \calM_g\longrightarrow\calA_g
\]
the Torelli map that sends a smooth projective curve of genus $g$ to its Jacobian.

The Torelli map is known to extend to a morphism
$\overline{\tau}_g:\overline{\calM}_g\to\Perf$
from the Deligne-Mumford compactification~$\overline{\calM}_g$ of~$\calM_g$ to $\Perf$ \cite{AlBr}, but for our purposes, it will suffice to work with the well-known partial extension
\[
\tau'_g:\calM'_g\longrightarrow\Part
\]
from the moduli space $\calM'_g$ of irreducible stable curves of genus $g$ with at most one node. The partial compactification $\calM'_g$ is
 the union of $\calM_g$
and the boundary divisor $\Delta'=\partial\calM'_g$ consisting of singular curves with only one node, which is non-separating.
It is well-known that $\Pic_\QQ(\calM'_g)=\QQ\lambda_1\oplus\QQ\delta'$
for $g\geq 3$, and that the map induced by $\tau'_g$ on Picard groups is
\[
(\tau_g')^*\lambda=\lambda_1,\quad (\tau_g')^*\delta=\delta'\, .
\]
The  slope for a divisor $a\lambda_1-b\delta'$ on $\calM'_g$ is defined to be $a/b$, and the slopes of cones of divisors on $\calM'_g$ are defined analogously to $\calA_g'$.

\begin{rem}\label{rem:FP}
The standard definition of slope for an effective divisor in $\overline{\calM}_g$
involves the vanishing order at all the boundary divisors of $\overline{\calM}_g$.
It follows from \cite{FP} that, if we limit ourselves
to divisors of slopes less than $29/3$ for $g\leq 5$,
then the two definitions are equivalent.
\end{rem}

As a consequence of the above discussion, we have obtained the following

\begin{lm}\label{lm:restrict-Mg}
Let $g\geq 4$ and let $E$ be an (internal) effective divisor on $\calA_g$.
\begin{itemize}
\item[(i)]
If $E$ does not contain the Jacobian locus $\calJ_g$, then $(\tau'_g)^{-1}(E)$ is an
effective divisor in $\calM'_g$ of slope $s(E)$;
\item[(ii)]
If $s(E)<s_\Eff(\calM_g)$, then $E$ contains the Jacobian locus $\calJ_g$.
\end{itemize}
\end{lm}

\medskip
Intersecting an effective divisor with the locus t $\calH_g\subset\calM_g$ of hyperelliptic Jacobians can also provide a constraint for the slope.
The closure $\calH'_g$ of $\calH_g$ inside $\calM'_g$ is obtained
by adding the locus $\pa\calH'_g$ consisting of curves with one
non-disconnecting node, obtained from smooth hyperelliptic curves
of genus $g-1$ by identification of two points that are
exchanged by the hyperelliptic involution, cf.~\cite{ch}.

Call the restriction of $\lambda_1$ to $\calH'_g$ from  $\overline{\calM}_g$
still by $\lambda_1$, and denote $\xi_0$ the class of $\pa\calH'_g$
(which is also the restriction of $\delta_0$ from $\overline{\calM}_g$).
It is known that $\Pic_\QQ(\calH'_g)$ has dimension $1$ and is generated by $\lambda_1,\xi_0$ with
the relation $(8g+4)\lambda_1 =g\xi_0$
(see \cite[Proposition 4.7]{ch}).
The map $\overline{\tau}_g$
restricts to $\calH'_g\to\Part$ and sends
$\pa\calH'_g$ to the boundary of $\Part$.

The following result was proven by Weissauer \cite{we86} (see \cite{sm} for details). Here we present a different argument.

\begin{prop}\label{prop:restrict-hyp}
For every $g\geq 3$, modular forms of slope strictly less than
$8+\tfrac{4}{g}$ must contain $\calH_g$.
\end{prop}
\begin{proof}
Let $F$ be a modular form on $\calA_g$ with class $[F]=a\lambda-b\delta$, and suppose that $F$ does not vanish on the entire
$\calH_g$. We want to show that $s(F)=\tfrac{a}{b}\geq 8+\tfrac{4}{g}$.

The pullback of $F$ on $\calH_g$
vanishes on an effective divisor $V$ of class $[V]=a\lambda_1-\beta\xi_0$ with $\beta\geq b$.
Using the relation in $\Pic_\QQ(\calH'_g)$,
we obtain
$[V]=\beta\left(\tfrac{a}{\beta}-\left(8+\tfrac{4}{g}\right)\right)\lambda_1$.

Consider now the double cover $C_t$
of $\PP^1$ branched at $\lambda_1,\dots,\lambda_{2g+1},t$,
and fix distinct $\lambda_1,\dots,\lambda_{2g+1},t_0$ such that
$C_{t_0}\notin V$.
Then $(C_t)_{t\in \PP^1}$ induces a map $\PP^1\to\calH'_g$,
whose image is a complete, irreducible curve $B\subset\calH'_g$ not contained in $V\cup\pa\calH'_g$.
It follows that $\deg_B(V)\geq 0$ and $\deg_B(\lambda_1)>0$,
and so $\tfrac{a}{\beta}-\left(8+\tfrac{4}{g}\right)\geq 0$.
The conclusion follows, since $s(F)\geq \tfrac{a}{\beta}$.
\end{proof}

Another way to prove \Cref{prop:restrict-hyp}
is to use the explicit modular form of slope $8+\tfrac{4}{g}$,
whose zero locus avoids $\calH_g$, constructed in \cite{sm}.

By \Cref{rem:FP}, for $g=3,4,5$ the content of \Cref{prop:restrict-hyp}
can be also recovered as a consequence of Corollary 3.27 in \cite{hm}.

\subsection{Case $g=3$}
The moduli space $\calA_3$ has a meaningful effective divisor,
namely (the closure of) the hyperelliptic locus $\calH_3$.
\begin{proof}[Proof of \Cref{maincor:mov} for $g=3$]
By \Cref{prop:restrict-hyp} a divisor in $\calA_3$ with slope smaller than $\tfrac{28}{3}$ must contain $\calH_3$.
This implies that the only effective divisor that could be of slope under $\tfrac{28}{3}$ is (the closure of) the hyperelliptic locus itself, and so
$s_\Mov(\Perf[3])\ge\tfrac{28}{3}$.
Since the closure of $\calH_3$ coincides with the theta-null divisor, we obtain from \eqref{eq:classTg}
\[
s(\calH_3)=s(T_3)=s(18\lambda-2\delta)=9<\tfrac{28}{3}\,.
\]
It follows that
$$
 s_\Eff(\Perf[3])=s(\calH_3)=9
 \quad\text{and}\quad
 s_\Mov(\Perf[4])\geq \tfrac{28}{3}\,.
$$
Since $T_3$ satisfies Condition \eqref{cond} by \Cref{prop:cond},
\Cref{thm:main} provides a modular form $\DIFF_{3,18}(T_3)$
of class $56\lambda-\beta\delta$ with $\beta\geq 6$. If $\beta\ge 7$, then the slope of $\DIFF_{3,18}(T_3)$ would be $s(\DIFF_{3,18}(T_3))\leq 56/7$, which is less than 9, contradicting the knowledge of effective slope. Thus $\beta=6$, and
$$
 s_\Mov(\Perf[3])\leq s(\DIFF_{3,18}(T_3))=\tfrac{56}{6}=\tfrac{28}{3}\,.
$$
This proves that the moving slope is equal to $s_\Mov(\Perf[3])=\tfrac{28}{3}$, and is realized by $\DIFF_{3,18}(T_3)$.
\end{proof}

There are also other constructions of Siegel modular forms
in $\calA_3$ of slope $\tfrac{28}{3}$:
\begin{itemize}
\item Let $M$ be the set of all octuplets of even characteristics that are cosets of some three-dimensional vector space of characteristics, and define
$$
\chi_{28}\coloneqq\sum_{\text{$M$ even octuplet coset}}\left(\frac{T_3}{\prod_{m\in M}\theta_m}\right)^2\,.
$$
This can be checked to be a modular form of class $[\chi_{28}]=28\lambda-3\delta$, see \cite{Tsu}. We verify that $\chi_{28}$ cannot be divisible by $T_3$, as otherwise $\chi_{28}/T_3$ would be a holomorphic cusp form of weight $10$, which does not exist by \cite{Tsu,LR}.
\item Alternatively, one can consider the forms
$$\chi_{140}\coloneqq\sum_{\text{$m$ even}} \left(\frac{T_3}{\theta_m}\right)^8\,,$$
which can be shown to have class $[\chi_{140}]=140\lambda-15\delta$.
We remark that the decomposable locus $\calA^{\dec}_3$ can be described by the equations $T_3=\chi_{140}=0$, since $\calA^{\dec}_3$ is simply the locus where at least two theta constants vanish.
Since $\calA_3^{\dec}$ has codimension $2$ within $\calA_3$,
this confirms that the forms $T_3$ and $\chi_{140}$ could not have a common factor.
\item Since $T_3$ satisfies Condition \eqref{cond} by \Cref{prop:cond},
and since $T_3$ and $\chi_{28}$ are square-free and without common factors,
\Cref{lm:scalar-bracket}(v) ensures that
the Rankin-Cohen bracket $[T_3,\chi_{28}]$ does not vanish identically along $\tnull$.
By \Cref{lm:scalar-bracket}, $[T_3,\chi_{28}]$ has weight $140$, and
vanishes to order $\beta\geq 15$ along the boundary. However, if it were to vanish to order 16 or higher, then its slope would be at most $\tfrac{140}{16}=8.75$, which is impossible since $s_\Eff(\Perf[3])=9$. Thus we must have $\beta=15$, so that
\[
[[T_3,\chi_{28}]]=140\lambda-15\delta
\]
is a Siegel modular form that also realizes $s_\Mov(\Perf[3])=\tfrac{140}{15}=\tfrac{28}{3}$.
\end{itemize}

\subsection{Case $g=4$}
The locus of Jacobians $\calJ_4$ is a divisor in $\calA_4$, which is known to be the unique effective divisor on $\Perf[4]$ of minimal slope, see \cite{smmodfour}. It is known that the effective slope of $\calM_4$ is $s_\Eff(\calM_4)=\tfrac{17}{2}$ (see \cite{gie} and \cite{FP}),

By Riemann's theta singularity theory,
theta divisors of Jacobians are singular, and in fact $\calJ_4=N_0'$.
Since $I_4=S_4$ and since $[I_4]=8\lambda-\delta$ as recalled in \Cref{sec:N0},
this reconfirms the equality
$$
 s_\Eff(\Perf[4])=s(I_4)=s(8\lambda-\delta)=8\,.
$$
As $I_4$ satisfies Condition \eqref{cond} by \Cref{prop:cond},
\Cref{thm:main} applied to $I_4$ produces
a modular form $\DIFF_{4,8}(I_4)$ not divisible by $I_4$,
of class $[\DIFF_{4,8}(I_4)]=34\lambda-\beta\delta$, with $\beta\ge 4$. Again, if $\beta$ were at least 5, the slope would be at most $34/5<8$, contradicting the effective slope, and thus we must have $\beta=4$.

\begin{proof}[Proof of \Cref{maincor:mov} for $g=4$]
From the above discussion, it follows that
$$
 s_\Mov(\Perf[4])\leq s(\DIFF_{4,8}(I_4))=s(34\lambda-4\delta)=\tfrac{17}{2}\,.
$$
On the other hand, \Cref{lm:restrict-Mg} implies that
any effective divisor in $\calA_4$
that does not contain the locus of Jacobians
has slope at least $s_\Eff(\calM_4)=\tfrac{17}{2}$.
It follows that $s_\Mov(\Perf[4])\geq \tfrac{17}{2}$.

We thus conclude that $s_\Mov(\Perf[4])=\tfrac{17}{2}=s(\DIFF_{4,8}(I_4))$.
\end{proof}

There are at least two other modular forms in $\calA_4$ that
realize the moving slope.
\begin{itemize}
\item The first one is $T_4$, whose class is $[T_4]=68\lambda-8\delta$ by~\eqref{eq:classTg}.
\item The second one is the Rankin-Cohen bracket $[I_4,T_4]$.
Since $I_4$ satisfies Condition \eqref{cond} by \Cref{prop:cond},
and since $I_4$ and $T_4$ are square-free and without common factors,
\Cref{lm:scalar-bracket} ensures that
$[I_4,T_4]$ does not vanish identically along $N'_0=\calJ_4$, and has class $306\lambda-\beta\delta$, with $\beta\geq 36$.
Since the effective slope of $\calA_4$ is $8$
and the moving slope of $\calA_4$ is $\tfrac{17}{2}$, we must have $\beta=36$ and
$$
s_\Mov(\Perf[4])=\tfrac{17}{2}=\tfrac{306}{36}=s([I_4,T_4])\,.
$$
\end{itemize}

As mentioned in the introduction, the pullback $\tau_g^*T_g$ via the Torelli map
gives $2\Theta_{\nl}$ on $\calM_g$, i.e.~$\sqrt{T_g}$ is not a modular form,
but its restriction to $\calJ_g$ is a Teichm\"uller modular form. For $g=4$ we exhibit a Siegel modular form that intersects $\calJ_4$ in
the divisor $\Theta_{\nl}$, with multiplicity $1$.

\begin{proof}[Proof of \Cref{cor:M4}]
Recall that $S_4=I_4$.
By \Cref{maincor:mov} for $g=4$ proven above,
$\DIFF_{4,8}(S_4)$ realizes the moving slope of $\calA_4$,
and so it does not contain the divisor of minimal slope,
namely the Schottky divisor.
Thus $\tau_4^*\DIFF_{4,8}(S_4)$ is an effective divisor on $\calM'_4$ of class $34\lambda_1-4\delta'$,
which thus realizes the effective slope $s_\Eff(\overline\calM_4)=\tfrac{17}{2}$.
Thus the pullbacks $\tau_4^*T_4$ and $\tau_4^*\DIFF_{4,8}(S_4)$
must have the same support. Since $[T_4]=2[\DIFF_{4,8}(S_4)]$,
we conclude that $\tau_4^*\DIFF_{4,8}(S_4)=\Theta_{\nl}$.
\end{proof}

\begin{rem}
For the sake of completeness, we recall that the moving slope of $\calM_4$ is $s_\Mov(\calM_4)=60/7$, see \cite{Fed}.
We can exhibit a modular form (analogous to $\chi_{140}$ for $g=3$) with this slope, namely
$$
\phi_{540}\coloneqq \sum_{\text{$m$ even}} \left(\frac{T_4}{\theta_m}\right)^8\,.
$$
The Siegel modular form $\phi_{540}$ has class $540\lambda-63\delta$,  see \cite{igusagen4}, and hence $\tau_4^*\phi_{540}$ gives an effective
divisor on $\calM'_4$ that realizes the moving slope $s_\Mov(\calM'_4)$.
Finally, we observe that both $T_4$ and $\phi_{540}$ have slope less than $9$, and
the equations $T_4=\phi_{540}=0$ define, set theoretically,
the hyperelliptic locus $\calH_4\subset\calM_4$, as discussed in \cite{Fp}.
\end{rem}

\subsection{Case $g=5$}
We recall that one of the main results of \cite{fgsmv} was the proof that the divisor $N'_0$ in $\Perf[5]$ has minimal slope:
$$
 s_\Eff(\Perf[5])=s(I_5)=s(108\lambda-14\delta)=\tfrac{54}{7}=7.714\dots
$$
Since $I_5$ satisfies Condition \eqref{cond} by \Cref{prop:cond},
by \Cref{thm:main}
\[
[\DIFF_{5,108}(I_5)]=542\lambda-\beta\delta \quad\text{with}\quad \beta\geq 70
\]
is a modular form that does not vanish identically on $N'_0$.

\begin{proof}[Proof of \Cref{cor:main}]
If $\beta\ge 71$, then the slope of $\DIFF_{5,108}(I_5)$ would be at most
$$
 542/71=7.633\dots<7.714\dots=54/7=s_\Eff(\Perf[5])\,,
$$
which is a contradiction. Thus $\beta=70$, and
\[
s_\Mov(\Perf[5])\leq s(\DIFF_{5,108}(I_5))=\tfrac{271}{35}= 7.742\dots
\]
\end{proof}

\subsection{Case $g=6$}\label{sec:genus6}
For genus $6$, the slope is bounded from below
$s(\Perf[6])\geq \frac{53}{10}$ by \cite{FV}.
Moreover, an interesting Siegel modular form
$\theta_{L,h,2}$ of class $14\lambda-2\delta$ was constructed in \cite{mss}, showing that $s(\Perf[6])\leq 7$ and that the Kodaira dimension of $\calA_6$ is non-negative.

\begin{proof}[Proof of \Cref{cor:genus6}]
In light of the classification of modular forms in low genus and weight
in \cite{CT} and \cite{CTweb},
in genus $6$ there are no cusp forms in weight $7,8,9,11,13$.
Now, in genus $6$ there are no Siegel modular forms of weight $2$
and we have seen above that
$s(\Perf[6])\geq \frac{53}{10}$.
Hence, the unique (up to multiple) cusp form of weight $10$ vanishes with multiplicity one along $D$ (and so does a possible cusp form in weight $6$).
As $\ord_D\theta_{L,h,2}=2$, the form
$\theta_{L,h,2}$ must thus be prime.

As for the second claim,
there are two possibilities:
\begin{itemize}
\item[(a)]
there exists a Siegel modular form of slope at most $7$,
not divisible by $\theta_{L,h,2}$: in this case, from \Cref{lm:properties} it follows that $s_\Mov(\Perf[6])\le 7$;
\item[(b)]
$\theta_{L,h,2}$ is the unique genus~$6$ Siegel modular form of slope~$7$ (up to taking powers): the claim then follows from \Cref{cor:main},
since $s(\DIFF_{6,14}(\theta_{L,h,2}))=7+\frac{2}{2\cdot 6}=\frac{43}{6}$ (as usual, if it happened that $\DIFF_{6,14}(\theta_{L,h,2})$ were to vanish to order strictly higher than $6\cdot 2=12$, then its slope would be at most $\tfrac{86}{13}<7$).
\end{itemize}
In either case, the result is proven.
\end{proof}

In the above case (a) the moduli space $\calA_6$ would have Kodaira dimension
at least $1$, in case (b) it would have Kodaira dimension $0$.

\section{Pluriharmonic differential operators}\label{sec:diffop}
In this section we introduce a suitable differential operator on the space of modular forms
and prove \Cref{thm:main} using  a general result of \cite{Ibukiyama}. Before introducing the relevant notions, we explain the outline of what is to be done.

We are looking for an operator $\DIFF_{g,a}$ that will map a genus $g$ Siegel modular form $F$ of weight $a$
to another modular form satisfying certain properties: more precisely,
$\DIFF_{g,a}(F)$ will be a polynomial in $F$ and its partial derivatives.
There are various motivations for looking for $\DIFF_{g,a}$ of such a form, which are discussed for the general setup for the problem in
\cite{Ibukiyama}, \cite{ibuzagier}, \cite{ehib}.

In our situation, motivated by the occurrence
of $\det(\partial F)$ in our treatment of
Rankin-Cohen bracket (see \Cref{rem:intrinsic} and \Cref{prop11}),
we will want $\DIFF_{g,a}(F)$ to restrict to $\det(\partial F)$ along the zero locus $\{F=0\}$.
Note that $\det(\pa F)$ is homogeneous in~$F$ of degree $g$,
in the sense that each monomial involves a product of $g$ different partial derivatives of~$F$, and moreover it is a purely $g$'th order differential operator, in the sense that each monomial involves precisely $g$ differentiations.
Hence, we will look for a $\DIFF_{g,a}$ that shares these two properties.

Besides $F\mapsto \det(\pa F)$, another operator with the above properties is $F\mapsto F^{g-1}(\det\pa)F$, where each monomial is $F^{g-1}$ multiplied by a suitable $g$'th order partial derivative of~$F$.
Of course $\DIFF_{g,a}(F)$ cannot be defined either
as $\det(\pa F)$ or as $F^{g-1}(\det\pa) F$, as these are not modular forms. But
a wished-for $\DIFF_{g,a}$ can be constructed explicitly:
in order to do so, we will use the general machinery of \cite{Ibukiyama},
which implies that a differential operator with constant coefficients maps
a non-zero modular form to a modular form if the corresponding polynomial
is pluriharmonic and
satisfies a suitable transformation property under the action of $\GL(g,\CC)$.

We now begin by reviewing the general notation, before stating a particular case of \cite[Thm.~2]{Ibukiyama} that allows the construction of $\DIFF_{g,a}$.

\subsection{Polynomials and differential operators}
Let $R_1,\dots,R_g$ be a $g$-tuple of $g\times g$ {\em symmetric} matrices, and denote the entries of $R_h$ by $(r_{h;ij})$. Denote
\[
\CC[R_1,\dots,R_g]\coloneqq\CC[\{r_{h;ij}\}]
\]
the space of polynomials in the entries of these matrices. The group $\GL(g,\CC)$
naturally acts by congruence on each symmetric matrix $R_h$,
and so on the space $\CC[R_1,\dots,R_g]$. For every integer $v\geq 0$ we denote by $\CC[R_1,\dots,R_g]_v\subset\CC[R_1,\dots,R_g]$ the vector subspace of those polynomials $P\in\CC[R_1,\dots,R_g]$ that satisfy
\[
P(AR_1A^t,\dots,AR_gA^t)=\det(A)^{v}P(R_1,\dots,R_g)
\]
for all $A\in\GL(g,\CC)$.

For every polynomial $Q\in\CC[R_1,\dots,R_g]$ we define
\[
Q_\pa\coloneqq Q(\pa_1,\dots,\pa_g),
\quad\text{where as usual}\ (\pa_h)_{ij}\coloneqq\frac{1+\delta_{ij}}{2}\frac{\pa}{\pa \tau_{h;ij}}\,.
\]
Such $Q_\pa$ is then a  holomorphic differential operator with constant coefficients
acting on holomorphic functions in the variables $\tau_{h;ij}$.
We further define the holomorphic differential operator $\calD_Q$ that sends
a $g$-tuple of holomorphic functions $F_1(\tau_1),\dots,F_g(\tau_g)$ on~$\HH_g$ to another holomorphic function on the Siegel space given by
\[
\calD_Q(F_1,\dots,F_g)(\tau)\coloneqq \left. Q_\pa(F_1,\dots,F_g)\right|_{\tau_1=\dots=\tau_g=\tau}\,.
\]
What this means is that applying each $\pa_h$ takes the suitable partial derivatives of $F_h$ with respect to the entries of the period matrix $\tau_h$, and then once the polynomial in such partial derivatives is computed, it is evaluated at the point $\tau_1=\dots=\tau_g=\tau$.

While the general theory of applying $\calD_Q$ to a $g$-tuple of modular forms is very interesting, we will focus on the case $F=F_1=\dots=F_g$, denoting then simply $\calD_Q(F)\coloneqq\calD_Q(F,\dots,F)$.

\begin{exa}\label{exa:poly}
It is immediate to check that the following polynomial (the general notation $\fR$ will be introduced in \Cref{sec:poly2} below)
\[
\fR(\underbrace{1,\dots,1}_{\text{ $g$ times}}):=g!\,\sum_{\sigma\in S_g}\sgn(\sigma) \prod_{j=1}^g r_{j;j,\sigma(j)}
\]
induces the differential operator $\calD_{\fR(1,\dots,1)}(F)=g!\,\det(\partial F)$. On the other hand, the polynomial
\[
\fR(g,\underbrace{0,\dots,0}_{\text{$g-1$ times}}\!\!):=\sum_{\sigma\in S_g}\sgn(\sigma) \prod_{j=1}^g r_{1;j,\sigma(j)}
\]
induces the differential operator $\calD_{\fR(g,0,\dots,0)}(F)=F^{g-1}(\det\pa)F$.
We stress that while each term of $\calD_{\fR(1,\dots,1)}$ is a product of $g$ first-order partial derivatives of $F$, each term of $\calD_{\fR(g,0,\dots,0)}$ is equal to $F^{g-1}$ multiplied by one $g$'th order partial derivative of~$F$.
\end{exa}

Our main result is the following \Cref{thm:precise}, which is a refined version of \Cref{thm:main}: indeed, to obtain \Cref{thm:main} from it, we just need to set
\[
\DIFF_{g,a}(F)\coloneqq {\textstyle\frac{1}{g!}}\calD_{Q_{g,a}}(F)
\]
for every modular form $F$ of genus $g\geq 2$ and weight $a\geq\frac{g}{2}$
(the constant factor $g!$ is introduced only for notational convenience).

In order to motivate the statement below,
recall that we want $\calD_{Q_{g,a}}(F)$ to be equal
to $g!\,\det(\pa F)$ modulo $F$. Since $g!\,\det(\pa F)=\calD_{\fR(1^g)}(F)$
as in \Cref{exa:poly},
and since $\fR(1^g)$
belongs to $\CC[R_1,\dots,R_g]_2$, it is rather natural to look for
$Q_{g,a}$ inside $\CC[R_1,\dots,R_g]_2$.

\begin{thm}\label{thm:precise}
For every $g\geq 2$ and every
$a\geq\frac{g}{2}$ there exists a polynomial $Q_{g,a}\in\CC[R_1,\dots,R_g]_2$
such that the following properties hold for every genus $g$ Siegel modular form $F$ of weight $a$:
\begin{itemize}
\item[(i)] $\calD_{Q_{g,a}}(F)$ is a Siegel modular form of weight $ga+2$;
\item[(ii)] if $\ord_D F=b$, then $\ord_D\calD_{Q_{g,a}}(F)\ge gb$;
\item[(iii)] the restriction of $\calD_{Q_{g,a}}(F)$ to the zero locus $\{F=0\}$ of~$F$ is equal to $g!\cdot\det(\pa F)$.
\end{itemize}
Moreover, for any other polynomial $Q'_{g,a}\in\CC[R_1,\dots,R_g]_2$ such that $\calD_{Q'_{g,a}}$ satisfies properties (i) and (iii), the difference $\calD_{Q_{g,a}}(F)-\calD_{Q'_{g,a}}(F)$ is a Siegel modular form divisible by~$F$.
\end{thm}

The above differential operator $\calD_{Q_{g,a}}$, which is homogeneous of degree $g$,
can be also applied to modular forms with a character, which only occur for $g=2$:
in this case, the output is a modular form (with trivial character).

\begin{rem}
As a consequence of \Cref{thm:precise}(iii),
if a modular form $F$ of genus $g$ and weight $a\geq\frac{g}{2}$ satisfies Condition \eqref{cond}, then $\calD_{Q_{g,a}}(F)$ does not vanish identically on the zero divisor of $F$.
\end{rem}

The reason we are able to construct $Q_{g,a}$ explicitly is that we can use a lot of prior work, especially by the second author and collaborators, on differential operators acting on modular forms.
In particular, by \Cref{thm:ibu} the operator $\calD_{Q_{g,a}}$ will map modular forms to modular forms if $Q_{g,a}$ is pluriharmonic --- this essential notion will be recalled in \Cref{sec:pluri}.

Thus to prove \Cref{thm:precise}, it will suffice to construct a pluriharmonic
$Q_{g,a}\in \CC[R_1,\dots,R_g]_2$. Property (i) will rely on \Cref{thm:ibu}
and (ii) will be easily seen to hold.
Up to rescaling, we will also check (iii), and the last claim will follow.

\subsection{A basis of $\bm{\CC[R_1,\dots,R_g]_2}$}\label{sec:poly2}
Consider the $g$-tuple of symmetric $g\times g$ matrices $R_1,\dots,R_g$. We set
\begin{equation}\label{eq:Mdefined}
\fR\coloneqq t_1R_1+\dots+t_gR_g\,,
\end{equation}
and denote by $\fR(\bfn)\in\CC[R_1,\dots,R_g]$ the coefficients of the expansion of the determinant
\[
\det(\fR)=\sum_{\bfn\in \bfN_g}\fR(\bfn)t_1^{n_1}\dots t_g^{n_g}
\]
as a polynomial in the variables $t_1,\dots,t_g$, where
$$\bfN_g\coloneqq\{\bfn=(n_1,\dots,n_g)\in\NN^g\,|\, n_h\geq0\textrm{ for all }h,\ \sum n_h=g\}\,.$$
The importance of the polynomials $\fR(\bfn)$ for us
relies on the fact that they clearly belong to $\CC[R_1,\dots,R_g]_2$,
simply because $\det(A\fR A^t)=\det(A)^2\det(\fR)$ for all $A\in \GL(g,\CC)$.

The following lemma, of a very classical flavor, is due to Claudio Procesi.
\begin{lm}\label{procesi}
The set of polynomials $\{\fR(\bfn)\}_{\bfn\in\bfN}$ is a basis of $\CC[R_1,\dots,R_g]_2$\,.
\end{lm}
\begin{proof}
Let $V$ be a complex $g$-dimensional vector space
and let $\GL(V)$ naturally act on $\Sym^2(V^*)^{\oplus g}$ via
\[
A\cdot ((\phi_1\otimes\phi_1),\dots,(\phi_g\otimes\phi_g)):=
((\phi_1A)\otimes (\phi_1 A),\dots,(\phi_g A)\otimes(\phi_g A))
\]
for $A\in\GL(V)$.
Consider the $\CC$-algebra $\calI(V, g)$ of
$\SL(V)$-invariants inside $\Sym^2(V^*)^{\oplus g}$.
The quotient  $\GL (V) / \SL (V) \cong\CC^*$ acts on $\calI(V,g)$
and, under this action, the algebra of invariants decomposes as
$$
\calI (V, g) =\bigoplus_d \calI (V, g) _d, \ \hbox{where}\
\calI (V, g)_d\coloneqq \{P\in \calI(V,g)\ |\ A\cdot P=\det(A)^{2d}P\}.
$$
Clearly $\calI (V, g)_d$ is simply the subspace of $\calI(V, g)$ consisting of invariant polynomial maps
$P:\Sym^2(V)^{\oplus g}\rightarrow\CC$ of total degree $d\cdot \dim(V)$
with respect to the above $\CC^*$-action.
The subspace $\calI(V,g)_1$ decomposes as
\[
\calI(V,g)_1=\bigoplus_{\bfn\in\bfN_g} \calI(V,g)_{\bfn}
\]
where $\calI(V,g)_{\bfn}\coloneqq\left(\bigotimes_{i=1}^g \Sym^2(V^*)^{\otimes n_i}\right)^{\SL(V)}$
denotes the subspace of invariant polynomial functions $\Sym^2(V)^{\oplus g}\rightarrow\CC$ of multi-degree $\bfn$.

Since it is easy to check that $\fR(\bfn)\in \calI(V,g)_{\bfn}$,
it is enough to show
that $\calI(V,g)_{\bfn}$ has dimension $1$ for all $\bfn\in\bfN_g$.
Moreover, $\calI(V,g)_{\bfn}$ is isomorphic to $\calI(V,g)_{(1,\dots,1)}=(\Sym^2(V^*)^{\otimes g})^{\SL(V)}$
as an $\SL(V)$-module for all $\bfn\in\bfN_g$, and so it is enough to show
that $(\Sym^2(V^*)^{\otimes g})^{\SL(V)}$ has dimension at most $1$
(and in fact it will have dimension $1$, since $\fR(\bfn)\not\equiv 0$).

Thinking of $(\Sym^2(V^*)^{\otimes g})^{\SL(V)}$
as a subspace of $((V^*)^{\otimes 2g})^{\SL(V)}$,
we describe a basis of $((V^*)^{\otimes 2g})^{\SL(V)}$;
it is enough to do that for $V=\CC^g$.

Let $\mathfrak{P}$ be the set of all
permutations $(I,J)=(i_1,\dots,i_g,j_1,\dots,j_g)$ of $\{1,2,\dots,2g\}$
such that $i_h<j_h$ for all $h=1,\dots,g$.
For every $(I,J)\in\mathfrak{P}$ we denote by
\[
[i_1,\dots,i_g][j_1,\dots,j_g]:V^{\otimes 2g}\longrightarrow\CC
\]
the linear map that sends $v_1\otimes\dots\otimes v_{2g}$
to $\det(v_I)\det(v_J)$, where $v_I$ is the matrix whose $h$-th column is $v_{i_h}$
(and similarly for $v_J$).
It is a classical fact that the collection of $[i_1,\dots,i_g][j_1,\dots,j_g]$
with $(I,J)\in\mathfrak{P}$ is a basis of $((V^*)^{\otimes 2g})^{\SL(V)}$.

Fix now $(I,J)$ and consider the restriction of $[i_1,\dots,i_g][j_1,\dots,j_g]$
to $\Sym^2(V)^{\otimes g}$, and in particular to the vectors
of type
\[
v_1\otimes v_1\otimes v_2\otimes v_2\otimes\cdots\otimes v_g\otimes v_g
\]
which generate $\Sym^2(V)^{\otimes g}$.
Note that $[i_1,\dots,i_g][j_1,\dots,j_g]$ is alternating both in $I$ and in $J$,
and vanishes on all vectors
$v_1\otimes v_1\otimes v_2\otimes v_2\otimes\cdots\otimes v_g\otimes v_g$
as soon as either $I$ or $J$ contains $\{2m-1,\,2m\}$ for some $m=1,\dots,g$.
It follows that all
elements $[i_1,\dots,i_g][j_1,\dots,j_g]$ of the above basis of
$((V^*)^{\otimes 2g})^{\SL(V)}$
vanish on $\Sym^2(V)^{\otimes g}$,
except possibly $[1,3,5,\dots,2g-1][2,4,6,\dots,2g]$.
We conclude that $(\Sym^2(V^*)^{\otimes g})^{\SL(V)}$ is at most $1$-dimensional.
\end{proof}

\subsection{Definition of the polynomial $\bm{Q_{g,a}}$}\label{sec:explicit}
In view of \Cref{procesi}, the sought polynomial $Q_{g,a}\in\CC[R_1,\dots,R_g]_2$
must be a linear combination of the $\fR(\bfn)$'s.
Now we give the explicit formulas for the coefficients of such a linear
combination. Pluriharmonicity of $Q_{g,a}$ will be verified in \Cref{sec:Qpluri}.

We define the constant
\[
C(1)\coloneqq (g-1)\prod_{i=1}^{g-1}(2a-i)\,.
\]
Moreover, for every $m=2,\dots,g$ we define the constant
\[
C(m)\coloneqq (-1)^{m-1}(m-1)!(2a)^{m-1}\prod_{i=m}^{g-1}(2a-i)\,,
\]
where for $m=g$ the last product above is declared to be equal to $1$, so that $C(g)=(-1)^{g-1}(g-1)! (2a)^{g-1}$.

Assume $2a\geq g\geq 2$, so that $C(1)\neq 0$, and
define then
\[
Q_{g,a}:=\frac{1}{C(1)}\sum_{\bfn\in\bfN_g} c(\bfn)\fR(\bfn)\,,
\]
where
\begin{enumerate}
  \item $c(1,\dots,1)\coloneqq C(1)$;
  \item if at least two of $n_1,\dots,n_g$ are greater than 1, then we set $c(\bfn)\coloneqq 0$;
  \item if $n_h=m>1$ for some $h$, while $0\le n_j\le 1$ for any $j\neq h$, then we set $c(\bfn)\coloneqq C(m)$.
\end{enumerate}
Hence $c(\bfn)\neq 0$ if and only if $\bfn$ is equal to $(m,1,1,\dots,1,0,0,\dots,0)$ for some $m\ge 1$,
up to permuting its components.

\subsection{Explicit formulas}

In order to have a more explicit expression for the polynomials $\fR(\bfn)$,
we expand the relevant determinants.

\begin{ntn}
If $M$ is a $g\times g$ matrix and $I,J\subset\{1,2,\dots,g\}$ with $|I|=|J|$, we denote by $M_{IJ}$ the minor of $M$
consisting of rows $I$ and columns $J$, and denote by $\det_{IJ}(M)$ the determinant of $M_{IJ}$
(if $|I|=|J|=0$, then we formally set $\det_{IJ}(M)\coloneqq 1$). Moreover, we let $\hI$ be the complement of $I$
and, if $i\in \{1,2,\dots,g\}$, then we let $\hat{\imath}:=\{1,2,\dots,g\}\setminus\{i\}$.
\end{ntn}

Applying the Laplace expansion several times yields
\begin{equation}\label{laplace}
\fR(\bfn)=
\sum_{I_\bullet,J_\bullet}
\epsilon(I_\bullet,J_\bullet)
\det_{I_1 J_1}(R_1)\cdots \det_{I_g J_g}(R_g)\,,
\end{equation}
where $(I_\bullet,J_\bullet)=(I_1,\dots,I_g,J_1,\dots,J_g)$ and
\begin{itemize}
\item $I_1, \dots, I_g, J_1, \dots, J_g$ run over all subsets of $\{1,\dots,g\}$ such that $|I_i|=|J_i|=n_i$ for each $i=1,\dots,g$ and $\sqcup_{i=1}^{g}I_i= \sqcup_{i=1}^{g}J_i=\{1,\dots,g\}$;
\item $\epsilon(I_\bullet,J_\bullet)$ is the signature of the element of $S_g$ that maps $(I_1,\dots,I_g)$ to $(J_1,\dots,J_g)$, where the elements inside each subset $I_i$ or $J_i$ are ordered from minimum to maximum.
\end{itemize}

In order to compute $\calD_{\fR(\bfn)}(F,\dots,F)$,
consider $\bfn=(m,1,\dots,1,0,\dots,0)$.
Regarding the partial sum of the parts for $|I_j|=1$ as
expansions of determinants by Laplace expansion, we have
\begin{equation}\label{diffresult}
\calD_{\fR(\bfn)}(F,\dots,F) =F^{m-1}\sum_{|I|=|J|=m}
\epsilon(I,J)(g-m)!(\det\nolimits_{IJ}\pa)F\cdot
\det\nolimits_{\widehat{I},\widehat{J}}(\pa F)\,,
\end{equation}
where we denote
\[
\epsilon(I,J)\coloneqq(-1)^{i_1+\dots+i_m+j_1+\dots+j_m}.
\]
Thus we have obtained the following.

\begin{cor}\label{cor:summands}
\hskip2mm

\begin{itemize}
\item[(i)]
If $\bfn=(1,\dots,1)$, then $\calD_{\fR(1,\dots,1)}(F)=g!\,\det(\partial F)$.
\item[(ii)]
If $\bfn\in\bfN_g$ and $\bfn\ne (1,\dots,1)$, then $\calD_{\fR(\bfn)}(F)$ is a multiple of $F$.
\end{itemize}
\end{cor}
\begin{proof}
For (i), note that
\[
\sum_{j=1}^{g}(-1)^{i+j}\pa_{ij}F\cdot \det\nolimits_{\widehat{\imath}\,\widehat{\jmath}}(\pa F)
=\det(\pa F)
\]
for every $i=1,\dots,g$. Then formula \eqref{diffresult} for $\bfn=(1,1,\dots,1)$ (that is, for $m=1$) yields
\[
\calD_{\fR(1,\dots,1)}(F)=
\sum_{i,j=1}^{g}(-1)^{i+j}\pa_{ij}F\cdot\det\nolimits_{\widehat{\imath}\,\widehat{\jmath}}(\pa F)
=g!\,\det(\pa F)\,.
\]
For (ii), we observe that since $\sum n_h=g$, and all $n_h\ge 0$, it follows that unless $\bfn=(1,\dots,1)$, there exists at least one $h$ such that $n_h=0$. But then the polynomial $\fR(\bfn)$ would contain no $r_{h;ij}$, which is to say that $F_h$ is not differentiated at all by $\calD_{\fR(\bfn)}(F_1,\dots,F_g)$. This finally means that $\calD_{\fR(\bfn)}(F_1,\dots,F_g)$ is divisible by $F_h$, and thus $\calD_{\fR(\bfn)}(F)$ is divisible by $F$.
\end{proof}

\begin{exa}
For $g=2,3$ we have
\begin{align*}
\calD_{Q_{2,a}}(F,F) &=2\det(\partial F)+\frac{2(2a)}{1-2a}F\cdot(\det\partial) F\,,\\
\calD_{Q_{3,a}}(F,F,F) &= 6\det(\partial F)+
 \frac{3(2a)^2}{(2a-1)(2a-2)} F^2\cdot (\det\partial) F\\
&\quad - \frac{3(2a)}{(2a-1)}
  F\sum_{i,j=1}^3(\partial_{ij}F)\cdot(\det\nolimits_{\hat{\imath}\hat{\jmath}}\partial)F\,.
  \end{align*}
\end{exa}

\subsection{Pluriharmonic polynomials}\label{sec:pluri}
By \Cref{thm:ibu} the most important step toward the proof of \Cref{thm:precise}
is checking that $Q_{g,a}$, defined in \Cref{sec:explicit} as a linear combination of $\fR(\bfn)$, is pluriharmonic.
In this section we recall the relevant setup, definitions, and statements.

Fix an $g\times k$ matrix $X=(x_{i\nu})$, and denote for $1\leq i,j\leq g$
\[
\Delta_{ij}(X)\coloneqq\sum_{\nu=1}^{k}\frac{\pa^2}{\pa x_{i\nu}\pa x_{j\nu}}\,.
\]
For a polynomial $P(R)$ in the entries of a symmetric $g\times g$ matrix $R=(r_{ij})$ we denote $\tP(X)\coloneqq P(X X^t)$.

\begin{df}
The polynomial $P$ is called {\em pluriharmonic} (with respect to $X$) if $\Delta_{ij}\tP=0$ for all $1\leq i,j \leq g$.
\end{df}

To detect this pluriharmonicity in terms of $R$, we define the differential operator in variables $(r_{ij})$ by
\begin{equation}\label{eq:Ddefined}
D_{ij}\coloneqq k\cdot\pa_{ij}+\sum_{u,w=1}^{g}r_{uw}\pa_{iu}\pa_{jw}\,,
\end{equation}
where $\pa_{ij}:=\frac{1+\delta_{ij}}{2}\frac{\pa}{\pa r_{ij}}$. Then a direct computation yields
\begin{equation}\label{eq:DDelta}
(D_{ij}P)(XX^t)=\Delta_{ij}(\tP(X))\,,
\end{equation}
where $P(r_{ij})$ is any polynomial, and the LHS means $D_{ij}$ is applied to $P$, and then evaluated at $XX^t$.

This equality shows that computing the $\Delta_{ij}$ derivative of $\tP$ (which is a second order differential operator) amounts to computing the $D_{ij}$ applied to~$P$, which is a differential operator that includes first and second order derivatives.  Thus pluriharmonicity is equivalent to the condition $D_{ij}(P)=0$ for all $1\le i,j\le g$.

Now the full setup we require is as follows. For a positive integer $k=2a$ we consider a $g$-tuple of $g\times k$ matrices $X_1,\dots,X_g$, and denote $R_h\coloneqq X_hX_h^t$.

\begin{df}
A polynomial $P\in\CC[R_1,\dots,R_g]$ is called {\em pluriharmonic} if
\[
\tP(X_1,\dots,X_g)\coloneqq P(X_1X_1^t,\dots, X_gX_g^t)\,,
\]
is pluriharmonic with respect to the $g\times (gk)$ matrix $X=(X_1,\dots,X_g)$.
\end{df}

The following result is a special case of \cite[Theorem 2]{Ibukiyama}, which shows the importance of pluriharmonicity.
\begin{thm}\label{thm:ibu}
For $g\geq 2$, let $P\in\CC[R_1,\dots,R_g]_2$ and let $F\neq 0$ be a Siegel modular form of genus $g$ and weight $a\geq\frac{g}{2}$.
Then $\calD_P(F,\dots,F)$ is a Siegel modular form of weight $ga+2$
if $P$ is pluriharmonic.
\end{thm}

Let us first give an elementary characterization of pluriharmonicity.
\begin{lm}\label{elementary}
Let $P\in\CC[R_1,\dots,R_g]$.
\begin{itemize}
\item[(i)] The polynomial $\tP(X)$ is pluriharmonic if and only if $\tP(AX)$ is harmonic
(i.e.~$\sum_{i=1}^{g}\Delta_{ii}\tilde{P}(AX)=0$) for any $A \in \GL(g,\CC)$.
\item[(ii)] Assume that $\tP\in\CC[R_1,\dots,R_g]_v$ for some $v$. Then
$\tP(X)$ is pluriharmonic if and only if $\Delta_{11}(\tP)=0$.
\end{itemize}
\end{lm}
\begin{proof}
The claim (i) is remarked in \cite{kashiwaravergne} and we omit the proof. In order to prove (ii), note that pluriharmonicity of $\tP$
implies that $\Delta_{11}(\tP)=0$ by definition.
Hence, it is enough to prove that $\Delta_{11}(\tP)=0$ implies pluriharmonicity.
For a fixed $i$ with $1\leq i\leq g$, let $A$ be the permutation matrix that exchanges the first row and the $i$-th row. Since
\begin{align*}
\Delta_{ii}(X)\cdot\tP(X) & =\det(A)^{-v} \Delta_{ii}(X)\cdot \tP(AX) \\
&=\det(A)^{2-v}\Delta_{11}(AX)\cdot \tP(AX)=0,
\end{align*}
the conclusion follows.
\end{proof}

Denoting $D_{h;11}$  the differential operator $D_{11}$ defined in \eqref{eq:Ddefined} with respect to the entries of the matrix $R_h$, and using \eqref{eq:DDelta} to rewrite $\Delta_{h,ij}$ for each $X_h$ as $D_{h;11}$,
by \Cref{elementary} we have the following.

\begin{cor}\label{cor:pluri}
Suppose that $\tP\in\CC[R_1,\dots,R_g]_v$ for some $v$.
Then $P$ is pluriharmonic with respect to the $g\times (gk)$ matrix $(X_1,\dots,X_g)$ if and only if
\begin{equation}\label{eq:Dharmonic}
\sum_{h=1}^{g}D_{h;11}P=0\,.
\end{equation}
\end{cor}

The above corollary applies to $Q_{g,a}$ and simplifies the verification of its pluriharmonicity.

\subsection{Pluriharmonicity of $Q_{g,a}$}\label{sec:Qpluri}

The result that we want to show is the following.

\begin{prop}\label{prop:existence}
The polynomial $Q_{g,a}$ is pluriharmonic.
\end{prop}

Since we will be dealing with minors of the matrix $\fR$ defined by~\eqref{eq:Mdefined}, we let
\begin{equation}\label{eq:bdefined}
\bfP\coloneqq\{\bfp=(n'_1,\dots,n'_g)\in\NN^g\,|\,n'_h\geq 0\textrm{ for all }h,\ \sum n'_h=g-1\}\,,
\end{equation}
and we denote by $\widehat{\fR}_{k;l}$ the determinant of the matrix $\fR_{\widehat{k};\widehat{l}}$, and denote by
$\widehat{\fR}_{k;l}(\bfp)$ the polynomial appearing in the expansion
\[
\widehat{\fR}_{k;l}=\sum_{\bfp\in\bfP}\widehat{\fR}_{k;l}(\bfp)
t_1^{p_1}\cdots t_g^{p_g}\,.
\]

We can now compute the derivative of $\fR(\bfn)$ that enters into the formula \eqref{eq:Dharmonic} for pluriharmonicity.
\begin{lm}\label{deriv}
For any $\bfn\in {\bf N}_g$, we have
\[
D_{h;11}\fR(\bfn)=2(k-n_h+1)\widehat{\fR}_{1;1}(\bfn-\bfe_h)\,,
\]
where $k=2a$, and  $\{\bfe_1,\dots,\bfe_g\}$ is the standard basis of $\ZZ^g$.
\end{lm}
\begin{proof}
By symmetry, it is enough to prove this for $h=1$; for simplicity, we just write $r_{ij}$ for the entries of the symmetric matrix $r_{1;ij}$, and define $\pa$ by \eqref{eq:padefined}.
We recall that $D_{1;11}=k\cdot\pa_{11}+\sum_{i,j=1}^g r_{ij}\pa_{1i}\pa_{1j}$.

Then by treating the cases $i=1$ and $i\ne 1$ separately, and checking the factor of $1/2$ versus $1$ appearing in the definition of $\pa$ for these entries, we see that
\[
\pa_{1i}\det (\fR)=2(-1)^{1+i}t_1\widehat{\fR}_{1;i}
\]
for any $i=1,\dots,g$. To compute the second order derivatives appearing in $D_{1;11}$, we first note that since $\widehat{\fR}_{1;1}$ does not depend on any $r_{1i}$,
we have $\pa_{1i}\pa_{1j}\det (\fR)=0$ if $i=1$ or $j=1$.

Otherwise, for $i,j\neq 1$, we compute
\begin{align*}
r_{ij}\pa_{1i}\pa_{1j}\det (\fR) &=(-1)^{1+j}t_1r_{ij}\pa_{1i}\widehat{\fR}_{1;j}
 =(-1)^{1+j+1+(i-1)}t_1^2r_{ij}\widehat{\fR}_{\{1,i\};\{1,j\}}\\
&=(-1)(-1)^{(i-1)+(j-1)}t_1^2r_{ij}\widehat{\fR}_{\{1,i\};\{1,j\}}\,.
\end{align*}
Summing these identities yields
\[
\sum_{i,j=2}^{g}r_{ij}\pa_{1i}\pa_{1j}\det (\fR)
=
(-t_1)\sum_{i,j=2}^{g}t_1r_{ij}(-1)^{(i-1)+(j-1)}\widehat{\fR}_{\{1,i\};\{1,j\}}\,.
\]
Here for a fixed $i$, the sum $\sum_{j=2}^{g}(-1)^{(i-1)+(j-1)}r_{ij}\widehat{\fR}_{\{1,i\};\{1,j\}}$
is nothing but the derivative of the $(i-1)$-th row of $\widehat{\fR}_{1;1}$ with respect to $t_1$, and thus
\[
\sum_{i,j=2}^{g}(-1)^{(i-1)+(j-1)}r_{ij}\widehat{\fR}_{\{1,i\};\{1,j\}}
=\frac{\pa}{\pa t_1}\widehat{\fR}_{1;1}\,.
\]
Recall that
$\widehat{\fR}_{1;1}=\sum_{\bfp\in\bfP}\widehat{\fR}_{1;1}(\bfp) t_1^{p_1}\dots t_g^{p_g} $ and note that
\[
t_1\frac{\pa}{\pa t_1}(t_1^{p_1}\dots t_g^{p_g}\widehat{\fR}_{1;1}(\bfp))
=p_1t_1^{p_1}\dots t_g^{p_g}\widehat{\fR}_{1;1}(\bfp)\,.
\]
Thus the coefficient of $t_1^{n_1}\dots t_g^{n_g}$ in the expansion of $D_{1;11}\det (\fR) $
is equal to
\[
2k\widehat{\fR}_{1;1}(n_1-1,n_2,\dots,n_g)-2(n_1-1)
\widehat{\fR}_{1;1}(n_1-1,n_2,\dots,n_g)\,.
\]
\end{proof}

As a consequence of \Cref{deriv}, we have
\begin{align*}
\frac{(2a-1)!}{(2a-g)!}\sum_{h=1}^g D_{h;11}Q_{g,a} &=
\sum_{h=1}^g c(\bfn) D_{h;11}\fR(\bfn)\\
&=\sum_{h=1}^g
2(k-n_h+1)c(\bfn)\widehat{\fR}_{1;1}(\bfn-\bfe_h)\\
&=2\sum_{h=1}^g (k-n'_h)c(\bfp+\bfe_h)\widehat{\fR}_{1;1}(\bfp)\,.
\end{align*}
Thus by \Cref{cor:pluri}, to check pluriharmonicity of $Q_{g,a}$ it is enough to check that
\begin{equation}\label{harmoniccondition}
  \sum_{h=1}^g (k-n'_h)c(\bfp+\bfe_h)=0
\end{equation}
for all $\bfp\in\bfP$.

Comparing the degrees of $\widehat{\fR}_{1;1}(\bfp)$ with respect to $R_h$, one can see that the set $\{\widehat{\fR}_{1,1}(\bfp)\,|\,\bfp\in\bfP\}$ is linearly independent over $\CC$, and so \eqref{harmoniccondition} is actually equivalent to the pluriharmonicity of $Q_{g,a}$.

\begin{proof}[Proof of \Cref{prop:existence}]
It is enough to verify \eqref{harmoniccondition} for every $\bfp\in\bfP$.
Up to reordering the entries of $\bfp$, we can assume that they are non-increasing.

If $n'_1\geq n'_2>1$, then $\bfp+\bfe_{\ell}$ has two entries larger than $1$, and so by definition we have $c(\bfp+\bfe_{\ell})=0$ for any $\ell$. It follows that all the terms in \eqref{harmoniccondition} are equal to zero, and the equation is trivially  satisfied.

For $\bfp=(1,\dots,1,0)$, the LHS of \eqref{harmoniccondition} is
$$
 k\cdot c(1,\dots,1)+(k-1)\biggl(c(2,1,\dots,1,0)+c(1,2,\dots,1,0)+\dots+c(1,\dots,2,0)\biggr) =k\cdot C(1)+(k-1)(g-1)C(2)\,.
$$
By definition of $C(1)$ and $C(2)$, the terms cancel, yielding $0$.

Let now $\bfp=(m,1,\dots,1,0,\dots,0)$ with $g-m$ entries $1$. If $2\leq \ell\leq g-m$, then $n_\ell>1$ and $c(\bfn'+\bfe_\ell)=0$ by definition.
We then have
$$c(\bfp+\bfe_1)=c(m+1,1,\dots,1,0,\dots,0)=C(m+1)\,.$$
If $g-m+1\leq \ell$, then $\bfp+\bfe_\ell$ is of type $(m,1,\dots,1,0,\dots,0)$, $(m,1, \dots,1,0,1,0,\dots,0)$, \dots, or $(m,1,\dots,1,0,\dots,0,1)$:
in all these cases $\bfp+\bfe_\ell$ has $g-m$ entries $1$, and thus $c(\bfp+\bfe_\ell)=C(m)$. So LHS of \eqref{harmoniccondition} is given by
\[
(k-m)C(m+1)+km\cdot C(m)\,,
\]
which also vanishes by our definition of the constants $C(m)$.
\end{proof}

\begin{exa}
In the case $g=2$ we obtain
\[
Q_{2,a}=  \fR(1,1)-\frac{2a}{2a-1}\fR(2,0)-\frac{2a}{2a-1}\fR(0,2)\,,
\]
where we have used $k=2a$.
This  is a special case of the discussion in \cite{ehib}.
\end{exa}

\begin{proof}[Proof of \Cref{thm:precise}]
The polynomial $Q_{g,a}$,
defined in \Cref{sec:explicit},
belongs to $\CC[R_1,\dots,R_g]_2$
and is pluriharmonic by \Cref{prop:existence}.
Then (i) and the first part of (ii) follow from \Cref{thm:ibu}.
Moreover, since $\calD_{Q_{g,a}}(F_1,\dots,F_g)$ is $\CC$-linear in each $F_h$,
and since $F_h$ and $\frac{\pa F_h}{\pa\tau_{ij}}$ have the same
vanishing order at the boundary for all $h$ and all $i,j$, it follows that
$\calD_{Q_{g,a}}(F,\dots,F)$ has vanishing order $\beta\geq gb$. This completes the proof of (ii).

As for (iii), note that $2a\geq g\geq 2$ ensures that the constant $C(1)$ defined in \Cref{sec:explicit} is non-zero and so, by construction,
\[
Q_{g,a}=\fR(1,\dots,1)+\sum_{\bfn\neq(1,\dots,1)} \frac{(2a-g)! c(\bfn)}{(2a-1)!(g-1)}\fR(\bfn)\,.
\]
By \Cref{cor:summands} it follows that
\[
\calD_{Q_{g,a}}(F,\dots,F)\equiv g!\ \det(\pa F)\quad\pmod{F}\,,
\]
and so (iii) is proven.
The last claim is an immediate consequence of (i) and (iii),
as the modular form $\calD_{Q_{g,a}}(F)-\calD_{Q'_{g,a}}(F)$ vanishes along $\{F=0\}$.
\end{proof}

We make one last remark on the above proof. We are not claiming that $Q_{g,a}$ or the associated differential
operator $\calD_{Q_{g,a}}$ are unique. Since we are looking for polynomials in $\CC[R_1,\dots,R_g]_2$,
these must be linear combinations of the $\fR(\bfn)$'s by \Cref{procesi}. If $Q'_{g,a}\in\CC[R_1,\dots,R_g]_2$ satisfies property (iii) in
\Cref{thm:precise}, then it must take the form
\[
Q'_{g,a}=\fR(1,\dots,1)+\sum_{\bfn\neq(1,\dots,1)} c'(\bfn) \fR(\bfn)
\]
by \Cref{cor:summands}. Hence, the restrictions of $\calD_{Q'_{g,a}}(F)$
and $\calD_{Q_{g,a}}(F)$ to the locus $\{F=0\}$ agree.
Note that the coefficients $c'(\bfn)$ may differ from the $c(\bfn)$ that were defined in \Cref{sec:explicit}.

\end{document}